\documentclass[11pt, a4paper, reqno]{amsart}
\DeclareMathAlphabet{\mathscrbf}{OMS}{mdugm}{b}{n}

\usepackage[utf8]{inputenc}
\usepackage[slovene, english]{babel}
\usepackage{polski}
\usepackage{amsmath,amsfonts,amssymb,mathrsfs,amsthm,amscd, ulem, stmaryrd}
\usepackage[pdftex]{graphicx}
\usepackage{tikz-cd}
\usepackage{soul} 
\usepackage{MnSymbol}
\usepackage[all,cmtip]{xy}
\usepackage{enumitem}

\usepackage{slashed}

\usepackage{pgfplots}

\usepackage[mathcal]{euscript}  

\usepackage{url}

\usetikzlibrary{arrows}

 \usepackage{relsize}

\tikzset{
  no line/.style={draw=none,
    commutative diagrams/every label/.append style={/tikz/auto=false}},
  from/.style args={#1 to #2}{to path={(#1)--(#2)\tikztonodes}}}

\title[Moduli of oriented formal groups and cellular motivic spectra]{
Moduli stack of oriented formal groups\\ and cellular  motivic spectra over $\mathbf C$
}
\author{Rok Gregoric}
\thanks{University of Texas at Austin}
\date{\today}
\subjclass[2020]{55N22, 14F42, 14A30}
\address{Department of Mathematics, University of Texas at Austin, Austin, TX 78712, USA}
\email{gregoric@math.utexas.edu}
\newtheorem{theorem}{Theorem}[subsection]
\newtheorem{theoremm}{Theorem}
\newtheorem{varianthm}[theorem]{Variant}

\newtheorem{corollary}[theorem]{Corollary}
\newtheorem{lemma}[theorem]{Lemma}
\newtheorem{prop}[theorem]{Proposition}

\theoremstyle{definition}
\newtheorem{definition}[theorem]{Definition}

\newtheorem{cons}[theorem]{Construction}

\newtheorem{ex}[theorem]{Examples}
\newtheorem{exun}[theorem]{Example}
\newtheorem{remark}[theorem]{Remark}
\newtheorem{remmark}[theoremm]{Remark}

\usepackage{tikz, calc}
\usetikzlibrary{matrix,arrows}

\newcommand*{\C}{{\mathbf C}}
\newcommand*{\Cat}{\mathcal C\mathrm{at}_\infty}

\newcommand*{\CAlg}{{\operatorname{CAlg}}}
\newcommand*{\CAlgcn}{{\operatorname{CAlg^{cn}}}}

\newcommand*{\mC}{\mathcal C}
\newcommand*{\mD}{\mathcal D}

\newcommand*{\mL}{\mathcal L}

\newcommand*{\mX}{\mathcal X}
\newcommand*{\mY}{\mathcal Y}

\newcommand*{\mS}{\mathcal S}
\newcommand*{\sL}{\mathscr L}
\newcommand*{\sA}{\mathscr A}

\newcommand*{\sO}{\mathcal O}
\newcommand*{\sF}{\mathscr F}

\newcommand*{\E}{\mathbb E_\infty}

\newcommand*{\heart}{\heartsuit}

\newcommand*{\sheafhom}{\mathscr{H}\kern -.5pt om}
\DeclareMathOperator{\Novak}{\mathscr{N}\text{\kern -3pt {\calligra\large ovak}}\,\,}

\usepackage{amsmath,calligra,mathrsfs}
\DeclareMathOperator{\fHom}{\mathscr{H}\text{\kern -3pt {\calligra\large om}}\,}
\DeclareMathOperator{\id}{\operatorname{id}}

\DeclareMathOperator{\Spfil}{\mathrm{Sp^{fil}}}
\DeclareMathOperator{\Sp}{\operatorname{Sp}}

\DeclareMathOperator{\Fun}{\operatorname{Fun}}

\DeclareMathOperator{\Spec}{\operatorname{Spec}}

\DeclareMathOperator{\Map}{\operatorname{Map}}
\DeclareMathOperator{\QCoh}{\operatorname{QCoh}}
\DeclareMathOperator{\IndCoh}{\operatorname{IndCoh}}

\DeclareMathOperator{\Ind}{\operatorname{Ind}}

\DeclareMathOperator{\Tot}{\operatorname{Tot}}

\DeclareMathOperator{\fib}{\operatorname{fib}}
\DeclareMathOperator{\cofib}{\operatorname{cofib}}

\DeclareMathOperator{\Sym}{\operatorname{Sym}}

\DeclareMathOperator{\MP}{\mathrm{MP}}

\DeclareMathOperator{\M}{\mathcal M^\mathrm{or}_\mathrm{FG}}
\DeclareMathOperator{\Mo}{\mathcal M^\heart_\mathrm{FG}}
\DeclareMathOperator{\Mn}{\mathcal M^{\mathrm{or}, \le n}_\mathrm{FG}}

\DeclareMathOperator{\GL}{\operatorname{GL}}
\DeclareMathOperator{\G}{\mathbf G}
\DeclareMathOperator{\Z}{\mathbf Z}
\DeclareMathOperator{\A}{\mathbf A}

\DeclareMathOperator{\Mod}{\operatorname{Mod}}

\renewcommand{\i}{\infty}

\newcommand{\w}{\widehat}

\renewcommand{\i}{\infty}
\renewcommand{\Mn}{\mathcal M^{\mathrm{or}, \le n}_\mathrm{FG}}

\DeclareFontFamily{U}{matha}{\hyphenchar\font45}
\DeclareFontShape{U}{matha}{m}{n}{
      <5> <6> <7> <8> <9> <10> gen * matha
      <10.95> matha10 <12> <14.4> <17.28> <20.74> <24.88> matha12
      }{}
\DeclareSymbolFont{matha}{U}{matha}{m}{n}

\DeclareMathSymbol{\varsubset}{3}{matha}{"80}

\usepackage{color}   
\usepackage[hypertexnames=false]{hyperref}
\hypersetup{
    linktoc=all,     
    linkcolor=black,  
}

\renewcommand{\i}{\infty}
\renewcommand{\o}{\otimes}

\usepackage{epigraph}

\begin{document}

\begin{abstract}
We exhibit a relationship between motivic homotopy theory and spectral algebraic geometry, based on the motivic
 $\tau$-deformation picture of Gheorghe, Isaksen, Wang, Xu.
More precisely, we identify cellular motivic spectra over $\mathbf C$ with  ind-coherent sheaves (in a slighly non-standard sense) on a certain spectral stack $\tau_{\ge 0}(\M)$. The latter is the connective cover of the non-connective spectral stack $\M$, the moduli stack of oriented formal groups, which we have introduced previously and studied in connection with chromatic homotopy theory. We also provide a geometric origin on the level of stacks for the observed $\tau$-deformation behavior on the level of sheaves, based on a notion of extended effective Cartier divisors in spectral algebraic geometry.
\end{abstract}
\maketitle

\section*{Introduction}

Motivic homotopy theory emerged around the dawn of the new millennium in the work of Morel and Voevodsky, e.g.\ \cite{Morel-Voevodsky}, \cite{Voevodsky1} \cite{Morel}, as well as Dundas, Hopkins, Levine, \O stv\ae r, R\"ondigs, Spitzweck, and others, as a setting where tools and ideas from homotopy theory could be applied to algebro-geometric objects.
It played a key role in solving a number of long-standing open questions in algebraic K-theory \cite{VoeK1}, \cite{VoeK2}.

Working over a fixed ground field $k$, motivic spaces are defined as certain kinds of space-valued sheaves on the site of smooth $k$-schemes, and motivic  spectra are obtained from (pointed) motivic spaces by inverting the operation of smashing with the projective line $\mathbf P^1_k$, viewed as an algebro-geometric analogue of the circle. It is hence clear that motivic spaces and motivic spectra receive canonical maps from usual spaces and spectra respectively. In the stable case, and for the ground field $k$ algebraically closed of characteristic zero, this turns out to be a fully faithful embedding by \cite[Theorem 1]{Levine}. We may therefore view motivic spectra as an algebro-geometric enlargement of usual homotopy-theoretic spectra.
Motivic homotopy theory is generally much more complicated than its classical non-motivic counterpart.
It is consequently quite surprising that the extra structure, present in the motivic setting, can in some cases be harnessed to aid in computations regarding classical homotopy theory. This is profitably exploited, with $k=\mathbf C$, by some of the most successful recent computational advances in understanding the (classical, i.e.\ non-motivic) stable homotopy groups of spheres \cite{Stable stems}, \cite{More stable stems}. 

These computational strides have inspired a substantial body of work, providing  conceptual explanations of the phenomena which underlie this uncanny success: \cite{Chow t}, \cite{DI},  \cite{GIKR}, \cite{GWX}, \cite{Krause} \cite{Pstragowski} to name but a few. Let us sketch the picture that emerged, which we will call the \textit{$\tau$-deformation picture},  also often known under the name \textit{cofiber of $\tau$ philosophy} in the literature.
 
 First we must fix the notation and terminology regarding motivic stable homotopy theory.
Let $\Sp_\C$ denote the motivic stable category, i.e.\ the $\i$-category of motivic spectra over the field of complex numbers $\C$. A characteristic aspect of motivic stable homotopy theory is that spheres in it are naturally bi-graded spheres $S^{t,w}$, with a \textit{topological degree} $t\in\Z$, and \textit{weight} $w\in \Z$. Under this indexing convention, the projective line $\mathbf P^1_{\mathbf C}$ coincides with the sphere $S^{2,1}$. For $t=w=0$, in which case the corresponding sphere is the unit for the smash product of motivic spectra, we prefer  the special notation $S_{\C} :=S^{0,0}$. The full subcategory $\Sp_\C^\mathrm{cell}\subseteq\Sp_\C$, generated by the spheres $S^{t,w}$ under colimits, is called \textit{cellular} (or \textit{Tate}) \textit{motivic spectra}. The association $X\mapsto X(\C)$, sending a smooth $\C$-scheme to the underlying homotopy type of the corresponding complex manifold, induces the \textit{Betti realization} functor $\Sp_{\C}\to\Sp$. This realization functor sends $S^{t,w}\mapsto S^t$, associating to motivic spheres the corresponding ordinary spheres.

With these conventions in place, the $\tau$-deformation picture may be encapsulated in the following statement:

\begin{theoremm}[{\cite{Gheorghe}, \cite{GWX}}]\label{The philosophy}
Everything in the following is implicitly $p$-completed, for an arbitrary prime $p$.
There exists a distinguished element $\tau\in \pi_{0, -1}(S_{\C})$ such that:
\begin{enumerate}[label =\textup{(\roman*)}]
\item \label{philun}
The Betti realization functor $\Sp_{\C}\to\Sp$ sends $\tau$ to the unit $1\in \Z\simeq \pi_0(S)$, and induces an equivalence of symmetric monoidal $\i$-categories
$$
(\Sp_{\C}^\mathrm{cell})^{\mathrm{Loc}(\tau)}\simeq \Sp
$$
between $\tau$-local cellular motivic spectra and the ordinary $\i$-category of spectra.

\item \label{phildeux}
The cofiber $S_{\C}/\tau = \cofib (S^{0, -1}\xrightarrow{\tau} S_\C)$ carries a canonical $\E$-algebra structure. There is a canonical equivalence of  symmetric monoidal $\i$-categories
$$
\Mod_{S_{\C}/\tau}(\Sp^\mathrm{cell}_{\C})\simeq \IndCoh(\Mo)
$$
of cellular motivic $S_{\C}/\tau$-module spectra and ind-coherent sheaves on the classical algebraic stack of formal groups $\Mo$.
\end{enumerate}
\end{theoremm}

Informally, the theorem may be summarized as asserting that the $\i$-category of cellular motivic spectra $\Sp_{\C}^\mathrm{cell}$ forms a $1$-parameter deformation, with deformation parameter $\tau$, whose \textit{generic fiber}, i.e.\ at $\tau^{-1}$, is the usual stable category, and whose \textit{special fiber}, i.e.\ $\tau=0$, is the $\i$-category of ind-coherent sheaves $\IndCoh(\Mo)$. Hence the name \textit{$\tau$-deformation picture}; see also \cite[Subsection 1.3, p.\,10]{BHS} for further discussion on this perspective.

\begin{remmark}
With the exception of \cite{BHS}, most of the literature on this subject does not make explicit mention of ind-coherent sheaves, instead preferring to view the special fiber at $\tau=0$  as  Hovey's stable category of $\mathrm{MU}_*\mathrm{MU}$-comodules. Since the two $\i$-categories are equivalent by \cite[Proposition 5.40]{Bartel Heard Valenzuela}, this is merely an aesthetic difference.
\end{remmark}

The $\tau$-deformation picture has been extended in many directions: in \cite{BHS} to work over the base $\mathbf R$ and relating to $\mathrm C_2$-equivariant spectra, in \cite{Pstragowski} to the integral (i.e.\ non-$p$-local) setting via synthetic spectra, and in \cite{Chow t} to an integral setting over any field $k$ 
 via the Chow $t$-structure. In \cite{GIKR}, an alternative approach to proving Theorem \ref{The philosophy} is given (also followed by \cite{BHS}), which relies on filtered spectra.

In this paper, we follow a suggestion of Morava and connect the $\tau$-deformation picture with spectral algebraic geometry. More specifically, we relate it to our prior work \cite{ChromaticCartoon}, where we examined how chromatic homotopy theory manifests in non-connective spectral algebraic geometry. The main object of interest there, building on the work of Lurie from \cite{Elliptic 2}, was \textit{the moduli stack of oriented formal groups} $\M$.  It is a non-connective spectral stack, with one of its characteristic properties being
 a Bott-isomorphism-like equivalence of quasi-coherent sheaves
$$
\beta:\Sigma^2(\omega_{\M})\simeq \sO_{\M},
$$
where $\sO_{\M}$ is the structure sheaf and $\omega_{\M}$ is the dualizing line of the universal oriented formal group. We also showed in \cite[Theorem 2.4.4]{ChromaticCartoon} that the global sections functor $\sF\mapsto\Gamma(\M; \sF)$ induces an equivalence of $\i$-categories
\begin{equation}\label{MTCCI}
\IndCoh(\M)\simeq \Sp
\end{equation}
between the stable $\i$-category of spectra and ind-coherent sheaves (in the non-standard sense of \cite[Definition 2.4.2]{ChromaticCartoon}, equivalent to Definition \ref{Def of IndCoh}) on the non-connective spectral stack $\M$.
 In this paper, the main object of interest will be its \textit{connective cover} $\tau_{\ge 0}(\M)$, studied in Section \ref{Section 1.4}. It is a (connective) spectral stack, obtained from $\M$ by affine-locally applying the $\E$-ring-level connective cover functor $A\mapsto \tau_{\ge 0}(A)$. On the level of $\tau_{\ge 0}(\M)$, there still exists the dualizing line of the universal formal group $\omega$, and a Bott map
$$
\beta :\Sigma^2(\omega)\to\sO_{\tau_{\ge 0}(\M)},
$$
which is however no longer an equivalence of quasi-coherent sheaves. With that setup,
here is the main result of this paper:

\begin{theoremm}[Theorem \ref{section 6}, Corollary \ref{dictionary}]\label{MTI}
Everything in the following statement is implicitly $p$-completed, for an arbitrary prime $p$.
There is a canonical equivalence of symmetric monoidal $\i$-categories
$$
\IndCoh(\tau_{\ge 0}(\M))\simeq \Sp_{\C}^\mathrm{cell}
$$
between ind-coherent sheaves on $\tau_{\ge 0}(\M)$ in the sense of Definition \ref{Def of IndCoh}, and cellular motivic spectra over $\C$. Under this equivalence, we have the following correspondence between motivic and sheaf-theoretic notions:
\begin{eqnarray*}
S_{\C} &\leftrightarrow & \sO_{\tau_{\ge 0}(\M)}\\
S^{2,1} &\leftrightarrow & \omega^{\o-1}\\
S^{t,w} &\leftrightarrow&  \Sigma^{t-2w} (\omega^{\o-w}) \\
S^{0, -1}\xrightarrow{\tau} S_\C &\leftrightarrow & \Sigma^2(\omega)\xrightarrow{\beta}\sO_{\tau_{\ge 0}(\M)}\\
S_\C[\tau^{-1}] &\leftrightarrow & \sO_{\M}\\
 S_\C/\tau &\leftrightarrow &  \sO_{\Mo}.
\end{eqnarray*}
\end{theoremm}

The proof of Theorem \ref{MTI} relies on the filtered module presentation of cellular motivic spectra from \cite[Section 6]{GIKR}, which we modify slightly in Section \ref{Section 1.3} to use the periodic complex bordism spectrum $\mathrm{MP}$, instead of its non-connective analogue $\mathrm{MU}$. We then give a similar filtered module presentation for the $\i$-category of ind-coherent sheaves on $\tau_{\ge 0}(\M)$ in Section \ref{Section 1.5}, and identify the two $\E$-algebra objects in Section \ref{section main}.
The main technical ingredient on the spectral algebraic geometry side is the observation of Lemma \ref{connective cover from beta} that the cofiber sequence
\begin{equation}\label{cofib intro}
\Sigma^2(\omega)\xrightarrow{\beta} \sO_{\tau_{\ge 0}(\M)}\to \pi_0(\sO_{\tau_{\ge 0}(\M)})\simeq \sO_{\mathcal M^\heart_\mathrm{FG}}
\end{equation}
 exhibits an equivalence with the $2$-connective cover
$
\tau_{\ge 2}(\sO_{\M})\simeq \Sigma^2(\omega)
$
in the $\i$-category of quasi-coherent sheaves on $\tau_{\ge 0}(\M)$.

From this perspective, the equivalences of $\i$-categories of Theorem \ref{The philosophy} are induced upon ind-coherent sheaves from the stack-level identifications:
\begin{enumerate}[label =(\roman*)]
\item The localization $\sO_{\tau_{\ge 0}}(\M)[\beta^{\pm 1}]\simeq \sO_{\M}$ recovers the structure sheaf of the non-connective stack $\M$, with the connective cover map $\M\to\tau_{\ge 0}(\M)$, which is an affine morphism. Using this, part \ref{philun} of Theorem \ref{The philosophy} reduces to  \eqref{MTCCI}.
\item By \eqref{cofib intro}, the cofiber $\sO_{\tau_{\ge 0}(\M)}/\beta$ is equivalent to the sheaf of $\E$-rings $\sO_{\Mo}$, the structure sheaf of the affine morphism of spectral stacks $\Mo\to \tau_{\ge 0}(\M)$. Part \ref{phildeux} of Theorem \ref{The philosophy} follows.
\end{enumerate}
In short, we may view the canonical cospan
$$
\M\to \tau_{\ge 0}(\M)\leftarrow \Mo
$$
as a deformation diagram of non-connective spectral stacks, with total space $\tau_{\ge 0}(\M)$, the deformation parameter $\beta$, the generic fiber (i.e.~$\beta$ invertible) being $\M$, and the special fiber (i.e.~$\beta= 0$) being $\Mo$.

We devote Part \ref{Part 2} to singling out the algebro-geometric circumstance that is responsible for this behavior. In Theorem \ref{defeorem}, we boil it down to the observation that the underlying ordinary stack map $\Mo\to\tau_{\ge 0}(\M)$ is an \textit{extended effective Cartier divisor} in the sense of Definition \ref{Def of effCart} (i.e.\ a closed immersion with invertible ideal sheaf), while the connective cover map $\M\to\tau_{\ge 0}(\M)$ is the complement of said divisor. This is as close as we know how to come with a stack-level deformation statement to the ideal ``fibration over $\mathbf A^1/\mathbf G_m$" picture, which would on the level of sheaves recover (and is suggested by) the  $\tau$-deformation picture. We discuss two separate issues with refining it in Remarks \ref{Remark problems 1}
and \ref{Remark problems 2}.

Though this deformation behavior of the stack $\tau_{\ge 0}(\M)$ is integral (i.e.\ not requiring an assumption of $p$-completeness), the comparison with motivic spectra of Theorem \ref{MTI} can only exist upon $p$-completion, else the motivic homotopy class $\tau\in \pi_{0, -1}(S_{\C})$ is not well-defined. On the other hand, \cite{Pstragowski} has proposed another explanation of the $\tau$-deformation picture, by introducing the algebraic notion of \textit{synthetic spectra}. These have since found other applications, for instance to asymptotic chromatic algebraicity \cite{Pstragowski2}, Goerss-Hopkins obstruction theory \cite{Pstragowski on GH}, and questions of manifold geometry in \cite{BHS:Manifolds}, \cite{BHS:Manifolds2}, and have been expanded in scope in \cite{PaPs}. Using the same filtered module spectra technique, which we use in the $p$-complete setting to prove Theorem \ref{MTI}, we obtain in Subsection \ref{section synthetic} an  integral comparison with synthetic spectra.

\begin{theoremm}[Theorem \ref{synthetic} and Remark \ref{MU over MP}]
There is a canonical equivalence of symmetric monoidal $\i$-categories
$$
\IndCoh(\tau_{\ge 0}(\M))\simeq \mathrm{Syn}_{\mathrm{MU}}^\mathrm{ev}
$$
between ind-coherent sheaves on $\tau_{\ge 0}(\M)$ and even synthetic spectra based on the  complex bordism spectrum $\mathrm{MU}$.
\end{theoremm}

We also consider in Variant \ref{variant} a finite-height analogue of this result. For any fixed height $n$ (and implicit prime $p$), it relates ind-coherent sheaves on the connective cover $\tau_{\ge 0}(\Mn)$ of the moduli stack $\Mn$ of oriented formal groups of height $\le n$ from \cite[Definition 3.2.1]{ChromaticFiltration}, and synthetic spectra based on the Lubin-Tate spectrum $E_n$.

In summation, the results of this paper demonstrate how (at least over the complex numbers and in the $p$-complete cellular context) stable motivic homotopy theory manifests in  spectral algebraic geometry, using the spectral stack $\tau_{\ge 0}(\M)$. Our prior work \cite{ChromaticCartoon}, \cite{ChromaticFiltration} shows that the same is true of chromatic homotopy theory, using the non-connective spectral stack $\M$. In this sense, present results may not seem so surprising; Voevodsky's conjecture \cite[Conjecture 9]{Voevodsky} (proved in \cite{Levine} by relying on work of Hopkins-Morel, see \cite{Hoyois}), and especially the related \cite[Theorem 4]{Levine} and \cite{Levine2}, already showcase a deep connection between the non-motivic Adams-Novikov spectral sequence, the progenitor of chromatic homotopy theory, and (the \textit{d\'ecalage} of) the slice filtration, one of the  fundamental constructions in motivic stable homotopy theory.

On the other hand, this may seem somewhat counter-intuitive (though not contradictory) in light of the results of \cite{Cisinski-Khan} about invariance of motivic homotopy under derived (or spectral) enhancement of the base. More precisely, \cite[Corollary 2.7.5]{Cisinski-Khan} shows that, for any quasi-compact quasi-separated spectral algebraic space $B$, the canonical functor between appropriately-defined $\i$-categories of motivic spectra $\Sp_B\to \Sp_{B^\heart}$, induced from sending a smooth spectral $B$-scheme $X$ to the ordinary smooth $B^\heart$-scheme $X^\heart$, is an equivalence of $\i$-categories. That means that one ready way of attempting to combine spectral algebraic geometry with motivic homotopy theory yields trivial results. We believe that the results of the present paper are instead a shadow of phenomena in the other possible combination of the two: \textit{motivic spectral algebraic geometry}. 
That should be roughly algebraic geometry, in which the role of affines is played by motivic $\E$-ring spectra.
In that case, Theorem \ref{MTI} suggests a close relationship between a motivic version of the stack $\M$, and the non-motivic connective cover $\tau_{\ge 0}(\M)$.
 We hope to return to this in future work.

The novel connection between  spectral algebraic geometry and motivic homotopy theory, established this paper, may lead to other new insights and applications. 
For instance, it suggests a connection between the motivic spectrum \textit{mmf} of motivic modular forms, defined  in a somewhat \textit{ad hoc} way using filtered spectra in \cite{GIKR}, and the connective cover $\tau_{\ge 0}(\mathcal M_\mathrm{Ell}^\mathrm{or})$ of the spectral moduli stack of oriented elliptic curves $\mathcal M_\mathrm{Ell}^\mathrm{or}$, which is connected to the usual spectrum of topological modular forms \textit{tmf} in \cite{survey} and \cite{Elliptic 2}. 
The relationship between preorientations of formal groups in the sense of Lurie and motivic homotopy theory has been touched on in the height $1$ case by \cite{Hornbostel: Preorientations}, but we hope that the connection with spectral algebraic geometry and the stack $\M$ that we demonstrate in this paper may shed some light on the fascinating interactions of chromatic and motivic ideas, as considered in various forms  and from various perspectives in \cite{Transchromatic},  \cite{Gheorghe: Exotic}, \cite{Hornbostel: Localizations}, \cite{Joachimi}, \cite{Krause}, \cite{AMG}, \cite{Motivic Landweber}, \cite{Sta2} etc.

\subsection*{Acknowledgments}
 Thanks to David Ben-Zvi and Andrew Blumberg for their constant support and encouragement, and to Viktor Burghardt, Tom Gannon, Paul Goerss, Elchanan Nafcha, and Sam Raskin for useful conversations about this project.

This paper was born out of a suggestion by Jack Morava that there should be a relationship between \cite{ChromaticCartoon} and the $\tau$-deformation picture of \cite{GWX} \textit{et al.}
 That came at much surprise to the author, who is therefore immensely grateful to Morava, both for his generosity in sharing this idea, as well as for his insistence that we give it serious thought.

\section{Comparison results}\label{Section 1}

\subsection{Conventions regarding motivic spectra}
Let $\Sp_{\C}$ denote the stable $\i$-category of motivic spectra over $\C$. The full subcategory of \textit{cellular motivic spectra} $\Sp_{\C}^\mathrm{cell}\subseteq\Sp_{\C}$ is taken to be the full subcategory generated under colimits by the bigraded spheres $S^{t, w}\simeq \Sigma^{t-w}(\Sigma^\infty(\G_m)^{\otimes w})$ for all $t,w\in\Z$. We denote $S_\C := S^{0, 0}$, often also denoted $\mathbf 1$ in the literature. It is the unit for the smash product symmetric monoidal structure on $\Sp_\C$.
The functor $X\mapsto X(\mathbf C)$, sending a smooth $\mathbf C$-scheme to the set of its $\mathbf C$-points, equipped with its analytic topology, gives rise to the \textit{Betti realization} functor $\Sp_\C\to\Sp$. Under Betti realization, bigraded spheres are sent to ordinary spheres as $S^{t,w}\mapsto S^t = \Sigma^t(S)$.

 If we  work in the $p$-completed setting for some prime $p$, there exists a special element $\tau\in \pi_{0, -1}(S_{\C})$ with several interesting properties - see \cite[Subsection 2.1]{BHS} for an explicit construction. Under Betti realization,  it corresponds to $1\in \mathbf Z\simeq \pi_0(S)$.

\begin{remark}
With \cite{GWX} as a prominent exception, most literature on stable motivic homotopy theory prefers the classical notation $\mathcal{SH}(\C)$ to what we call $\C$. That name is chosen  to comply with the convention that $\mathcal{SH}$ denotes usual non-motivic spectra. On the other hand, we follow \cite{HA} (and much of the rest of the modern homotopy-theoretic literature) in denoting the $\i$-category of spectra by $\Sp$, and so $\Sp_{\C}$ for motivic spectra seems more appropriate. Another perhaps more objective reason to prefer the notation $\Sp_{\C}$ over $\mathcal{SH}(\C)$ is that the latter suggests a homotopy category, whereas we mean the full $\i$-category instead.
\end{remark}

\begin{remark}
Another name in the $\tau$-deformation literature, that may differ from the standard conventions in the field of the motivic homotopy theory, is that of cellular spectra. This notion was introduced, and thoroughly studied in analogy with other homotopy-theoretic incarnations of cellularity, in \cite{Dugger-Isaksen}. On the other hand, much of the motivic literature, including \cite{BHS},  prefers the more traditional algebro-geometric term \textit{Tate motivic spectra}.  This has to do with the fact that the spheres $S^{0, n}$ for $n\in \Z$, which generate $\Sp_{\C}^\mathrm{cell}$ under limits, colimits and extensions, encode in motivic cohomology analogous structure to Tate twists in \'etale cohomology. 
\end{remark}

\subsection{Conventions regarding filtered spectra}
In what follows, the $\i$-category of \textit{filtered spectra} is defined as $\Spfil\simeq\Fun(\mathbf Z_\mathrm{poset}^\mathrm{op}, \Sp)$, which is to say, contravariant functors from the poset $\mathbf Z$ to spectra. A filtered spectrum $X$ therefore roughly consists of a diagram
$$
\cdots \to X_{n+1}\to X_n\to X_{n-1}\to\cdots
$$
in the $\i$-category $\Sp$. The $\i$-category of filtered spectra $\Spfil$ acquires a symmetric monoidal structure through Day convolution, which is explicitly given filtered-degree-wise by
$$
(X\o Y)_n \simeq \varinjlim_{n\le i+j} X_i\o Y_j
$$
in terms of the smash product of spectra, for all $n\in \Z$. For more details on filtered spectra, see \cite[Section 3]{Lurie Rotation}, \cite[Section 2.1]{GIKR}, or, in a more general context, \cite[Appendix B]{BHS}.

\subsection{Cellular motivic spectra as filtered module spectra}\label{Section 1.3}
The following is little more than a recap of \cite[Section 6]{GIKR}, the only change being that we use a periodic version of the algebraic bordisms spectrum.
Everything in this section is implicitly $p$-completed, for an arbitrary fixed prime $p$.

\begin{remark}
The $\i$-category of motivic spectra $\Sp_\C$ is presentable and stable, and therefore by \cite[Proposition 4.8.2.18]{HA} is tensored over the $\i$-category of spectra $\Sp$.  This allows us to form the mapping spectrum $\mathrm{map}_{\Sp_\C}(X, Y)\in \Sp$ for any pair of motivic spectra $X, Y\in \Sp_\C$, whose underlying infinite loop space $\Omega^\i(\mathrm{map}_{\Sp_\C}(X, Y))\simeq \Map_{\Sp_\C}(X, Y)$ recovers the usual mapping space.
\end{remark}

\begin{definition}
Given a motivic spectrum $X\in \Sp_{\C}$, we define its \textit{stable homotopy groups} to be the filtered spectrum $\Omega^{0, *}_\mathrm{st}(X)\in \Spfil$, defined as follows. Its $n$-th filtration level is given by the mapping spectra
$$
\Omega^{0, n}_\mathrm{st}(X)\simeq \mathrm{map}_{\mathrm{Sp}_\C}(S^{0, n}, X),
$$
and its filtration structure $\Omega^{0, n}_\mathrm{st}(X)\to \Omega^{0, n-1}_\mathrm{st}(X)$ is induced on mapping spectra by the map $S^{0, n-1}\xrightarrow{\tau} S^{0, n}$.
\end{definition}

The functor $\Omega^{0, *}_\mathrm{st} : \mathrm{Sp}_{\C}\to\Spfil$ is lax symmetric monoidal by \cite[Proposition 6.6]{GIKR}. It sends the monoidal unit $S_{\C}:=S^{0, 0}$ of motivic spectra into an $\E$-algebra object $\Omega^{0, *}_\mathrm{st}(S_{\C})\in \CAlg(\Spfil)$. The following is the main comparison result, expressing cellular motivic spectra in terms of filtered spectra:

\begin{theorem}[{\cite[Theorem 6.12]{GIKR}}]\label{Theorem Ctau}
The functor $\Omega^{0, *}_\mathrm{st}$ exhibits an equivalence of symmetric monoidal $\i$-categories
$
\Sp_\mathrm{\C}^\mathrm{cell}\simeq \Mod_{\Omega^{0, *}_\mathrm{st}(S_{\C})}(\Spfil).
$
\end{theorem}

To utilize this theorem, we must express the $\E$-algebra $\Omega^{0, *}_\mathrm{st}(S_\C)$ in a way that does not directly invoke motivic homotopy theory. For this purpose, let $\MP_{\C}:=\bigoplus_{i\in \mathbf Z}\Sigma^{2i, i}(\mathrm{MGL})$ be the motivic periodic bordism spectrum. We equip it with a standard $\E$-structure, e.g.\ either the one coming from the motivic Thom spectrum presentation
$$
\MP_\C\simeq \varinjlim (\Omega^\infty(\mathrm K(\C))\xrightarrow{S_{\C}^-}\Sp_{\C}),
$$
or the (in light of the analogous non-motivic result of \cite{Hahn-Yuan}, likely distinct) one coming from the equivalence $\MP_\C\simeq \Sigma^\infty_+(\mathrm{BGL})[\beta^{-1}]$ of motivic ring spectra of \cite{Gepner-Snaith}.

\begin{lemma}\label{cosimplicial filtered level}
There is a canonical equivalence $\Omega^{0, *}_{\mathrm{st}}(\MP_{\C}^{\otimes \bullet +1})\simeq \tau_{\ge 2*}(\MP^{\otimes \bullet +1})$ of cosimplicial objects in filtered spectra .
\end{lemma}

\begin{proof}
Recall from \cite[Lemma 6.7]{GIKR} (see also Remrk \ref{don't worry about odd primes} the analogous non-periodic result, i.e.\ the equivalence
\begin{equation}\label{lemma in the non-periodic case}
\Omega^{0, *}_\mathrm{st}(\mathrm{MGL}^{\o\bullet +1})\simeq \tau_{\ge 2*}(\mathrm{MU}^{\o\bullet +1}).
\end{equation}
From this, we obtain a canonical series of equivalences
\begin{eqnarray*}
\Omega^{0, *}_\mathrm{st}(\MP_{\C})&\simeq & \mathrm{map}_{\Sp_\C}\big(S^{0, *}, \bigoplus_{i_0, \ldots, i_\bullet\in \Z} \Sigma^{2(i_0+
\cdots + i_\bullet), (i_0 + \cdots +i_\bullet)}(\mathrm{MGL}^{\otimes\bullet +1})\big)\\
&\simeq &  \bigoplus_{i_0, \ldots, i_\bullet\in \Z} \Sigma^{2(i_0+
\cdots + i_\bullet)} \mathrm{map}_{\Sp_\C}(S^{0, *-(i_0+\cdots + i_\bullet)}, \mathrm{MGL}^{\otimes\bullet +1})\\
&\simeq &  \bigoplus_{i_0, \ldots, i_\bullet\in \Z} \Sigma^{2(i_0+
\cdots + i_\bullet)} \tau_{\ge 2(*-(i_0+\cdots + i_\bullet))}  (\mathrm{MU}^{\otimes\bullet +1})\\
&\simeq &  \tau_{\ge 2*}\big( \bigoplus_{i_0, \ldots, i_\bullet\in \Z} \Sigma^{2(i_0+
\cdots + i_\bullet)}  (\mathrm{MU}^{\otimes\bullet +1})\big)\\
&\simeq &\tau_{\ge 2*}(\MP^{\otimes \bullet +1}).
\end{eqnarray*}
Here the second equivalence follows from compactness of the bigraded spheres in $\Sp_{\C}$, the third is the result mentioned above, and the fourth is a consequence of connective covers commuting with direct sums due to compactness of the spheres. The first and the last equivalences are a matter of definitions.
\end{proof}

\begin{remark}\label{don't worry about odd primes}
Equivalence \eqref{lemma in the non-periodic case} is the key technical result underlying this paper. It is proved in \cite[Lemma 6.7]{GIKR} for $p=2$, but the argument given there works for all primes $p$, as shown also in \cite[Section 7.4]{Pstragowski}. Indeed, as given in \cite{GIKR}, the argument requires two technical inputs:
\begin{enumerate}[label = (\alph*)]
\item  The map $\Map_{\Sp_{\C}}(S^{0,a}, S^{0, b})\to \Map_{\Sp}(S^0, S^0)$, induced by the Betti realization, is a homotopy equivalence for all $a\le b$.\label{point one of this remark}
\item There is a bigraded ring isomorphism $\pi_{*, *}(\mathrm{MGL})\simeq\mathbf Z[\tau][t_1, t_2, \ldots]$ with $|t_i| = (2i, i)$, compatible via Beti realization with Milnor's famous graded-ring isomorphism $\pi_*(\mathrm{MU})\simeq \mathbf Z[t_1, t_2, \ldots]$ with $|t_i|=2i$.
\label{second point of this remark}
\end{enumerate}
For the second point, recall that everything is implicitly $p$-completed. To justify these two points,  \cite{GIKR} invokes the computations of \cite{GI16} for \ref{point one of this remark}, and \cite[Theorem 7]{HKO11} for \ref{second point of this remark}. Both of these references work exclusively in the setting of $p=2$, but the results extent to arbitrary prime $p$. For \ref{second point of this remark}, see \cite[Section 2.5]{Sta16}, and \ref{point one of this remark} follows from that using the same motivic Adams-Novikov spectral sequence argument as used in \cite{GI16}.
\end{remark}

\begin{lemma}\label{comparison lemma}
There is a canonical equivalence $\Omega^{0, *}_\mathrm{st}(S_{\C})\simeq \Tot(\tau_{\ge 2*}(\MP^{\otimes \bullet+1}))$ of $\E$-algebra objects in filtered spectra.
\end{lemma}

\begin{proof}
Just as in the non-motivic case, the convergence of the motivic Adams-Novikov spectral sequence implies that
$
S_{\C}\simeq \mathrm{Tot}(\MP_{\C}^{\otimes \bullet +1}),
$
i.e.\ the motivic sphere spectrum $S_{\C}$ (also denoted $S^{0, 0}_{\C}$ or $\mathbf 1$) is equivalent to the totalization of the Amitsur complex (or cobar construction) of $\MP_\C$. This implies that
\begin{eqnarray*}
\Omega^{0, *}_\mathrm{st}(S_{\C})&\simeq & \mathrm{map}_{\Sp_\C}(S^{0, *}, \mathrm{Tot}(\MP_{\C}^{\otimes \bullet +1})) \\
&\simeq & \Tot(\mathrm{map}_{\Sp_\C}(S^{0, *}, \MP_{\C}^{\otimes \bullet +1})) \\
&\simeq & \Tot(\tau_{\ge 2*}(\MP^{\otimes \bullet+1})),
\end{eqnarray*}
where the last equivalence is an application of Lemma \ref{cosimplicial filtered level}. That this equivalence of filtered spectra respects the $\E$-ring structures on both sides follows just like \cite[Proposition 3.7]{GIKR}.
\end{proof}

\subsection{The connective cover of the moduli stack of oriented formal groups}\label{Section 1.4}
Recall the stack of oriented formal groups $\M$ from \cite{ChromaticCartoon}. We showed in \cite[Corollary 2.3.7]{ChromaticCartoon} that it admits a simplicial presentation
$$
\M\simeq \left |\Spec(\MP^{\o \bullet +1})\right|.
$$
Hence it follows by \cite[Corollary 1.3.7]{ChromaticCartoon} that its connective cover is given by
$$
\tau_{\ge 0}(\M)\simeq \left|\Spec(\tau_{\ge 0}(\MP^{\o \bullet +1}))\right|.
$$

\begin{remark}
The spectral stack $\tau_{\ge 0}(\M)$ appears in \cite[Example  9.3.1.8]{SAG}, denoted there by $\CMcal{FG}^\mathrm{der}$,  under the name \textit{derived moduli stack of formal groups}. In spite of this suggestive terminology (which we reserve for the spectral stack $\mathcal M_\mathrm{FG}$, classifying formal groups in spectral algebraic geometry; see Section \ref{Section 1.5}), we are currently unaware of a convenient description of the functor of points of this stack, i.e.\ what kind of (preoriented) formal groups does it classify. This is in sharp contrast with the non-connective stack $\M$, which we know from \cite{ChromaticCartoon} to classify oriented formal groups.
\end{remark}

The $\i$-categories of quasi-coherent sheaves on the above two spectral stacks are therefore given by
$$
\QCoh(\M)\simeq \Tot(\MP^{\o \bullet +1}), \qquad \QCoh(\tau_{\ge 0}(\M))\simeq \Tot(\tau_{\ge 0}(\MP^{\o \bullet +1})).
$$
Recall from \cite[Subsection 1.5]{ChromaticCartoon} that there exists $k$-connective cover functors
$$
\tau_{\ge k} : \QCoh(\M)\to\QCoh(\tau_{\ge 0}(\M))
$$
for all $k\in \Z$, such that $\tau_{\ge 0}(\sO_{\M})\simeq \sO_{\tau_{\ge 0}(\M)}$. If a quasi-coherent sheaf $\sF\in\QCoh(\M)$ is presented as a compatible system of $\MP^{\o \bullet +1}$-module spectra $M_\bullet$, then the $k$-connective cover $\tau_{\ge k}(\sF)\in \QCoh(\tau_{\ge 0}(\M))$ is presented by the compatible system of $\tau_{\ge 0}(\MP^{\o \bullet+1})$-module spectra $\tau_{\ge k}(M_\bullet)$.

The following object will play one of the main roles in this paper, hence we give it a particularly simple name.

\begin{definition}
Let $\omega_{\M}\in\QCoh(\M)$ denote the dualizing line of the universal oriented formal group on $\M$, and by
$$
\omega :=\tau_{\ge 0}(\omega_{\M})\in\QCoh(\tau_{\ge 0}(\M))
$$
its connective cover. 
\end{definition}

By definition of an orietation of a formal group, there is a canonical equivalence of quasi-coherent sheaves
$\Sigma^2(\omega_{\M})\simeq \sO_{\M}$ on $\M$. Such an equivalence can of course not exist for the corresponding sheaves on $\tau_{\ge 0}(\M)$, and the following result shows precisely how much it fails.

\begin{lemma}\label{connective cover from beta}
There is a canonical cofiber sequence
$$
\Sigma^2(\omega)\xrightarrow{\beta} \sO_{\tau_{\ge 0}(\M)}\to \pi_0(\sO_{\tau_{\ge 0}(\M)})\simeq \sO_{\mathcal M^\heart_\mathrm{FG}}
$$
of quasi-coherent sheaves on $\tau_{\ge 0}(\M)$, which exhibits an equivalence
$$
\tau_{\ge 2}(\sO_{\M})\simeq \Sigma^2(\omega)
$$
in the $\i$-category  $\QCoh(\tau_{\ge 0}(\M))$.
\end{lemma}

\begin{proof}
The second part, along with  the existence of the desired map of quasi-coherent sheaves $\beta$, follows from the series of canonical equivalences
$$
\Sigma^2(\omega)\simeq \Sigma^2(\tau_{\ge 0}(\omega_{\M}))\simeq \tau_{\ge 2}(\Sigma^2(\omega_{\M}))\simeq \tau_{\ge 2}(\sO_{\M})
$$
in $\QCoh(\tau_{\ge 0}(\M))$. To identify the cofiber of $\beta$, it suffices to observe that, upon pullback along the faithfully flat cover $\Spec(\tau_{\ge 0}(\MP))\to \tau_{\ge 0}(\MP)$, we obtain the cofiber sequence
$$
\Sigma^2(\tau_{\ge 0}(\MP))\xrightarrow{\beta}\tau_{\ge 0}(\MP)\to \pi_0(\tau_{\ge 0}(\MP))\simeq L.
$$
The same holds for all $\tau_{\ge 0}(\MP^{\otimes \bullet +1})$, and the cosimplicial commutative ring $\pi_0(\MP^{\otimes \bullet +1})$ is the (cosimplicially-extended) Hopf algebroid presentation $(L\rightrightarrows W)$ of the stack of oriented formal groups $\Mo$.
\end{proof}

\begin{remark}
Since $\tau_{\ge 2}(\sO_{\M})\simeq \tau_{\ge 2}(\sO_{\tau_{\ge 2}(\M)})$, we could have expressed the conclusion of Lemma \ref{connective cover from beta} purely in terms of $\tau_{\ge 0}(\M)$, without reference to the non-connective spectral stack $\M$. On the other hand, the statement we gave instead has the benefit of generalizing to $n$-connective covers for all $n\in\Z$ in  Proposition \ref{higher connective covers}, instead of just $n\ge 0$.
\end{remark}

We wish to extend the result of Lemma \ref{connective cover from beta} to smash powers of $\omega$ - see Proposition \ref{higher connective covers}. This will first require a brief discussion of flat quasi-coherent sheaves in non-connective spectral algebraic geometry.

\begin{definition}
For any non-connective spectral stack $X$, we say that a quasi-coherent sheaf $\sF\in\QCoh(X)$ is \textit{flat} if, for any map of the form $f:\Spec(A)\to X$, the sheaf $f^*(\sF)$ corresponds to a flat $A$-module in the sense of \cite[Definition 7.2.2.10]{HA} under the equivalence of $\i$-categories $\QCoh(\Spec(A))\simeq \Mod_A$.
Let $\QCoh^\flat(X)\subseteq \QCoh(X)$ denote the full subcategory spanned by flat quasi-coherent sheaves.
\end{definition}

Since relative smash product preserves flatness by \cite[Proposition 7.2.2.16, (1)]{HA}, it follows that for any map of non-connective spectral stacks $f:X\to Y$, the quasi-coherent pullback functor $f^*:\QCoh(Y)\to\QCoh(X)$ restricts to a functor between the subcategories of flat quasi-coherent sheaves $f^* :\QCoh^\flat(Y)\to\QCoh^\flat(X)$.

\begin{ex}\label{example}
The dualizing line $\omega_{\w{\G}}$ of any formal group $\w{\G}$ over an $\E$-ring $A$ is a flat $A$-module by definition. This implies that the dualizing line of the universal formal group $\omega_{\mathcal M_\mathrm{FG}}$ is a flat quasi-coherent sheaf on the spectral moduli stack of formal groups $\mathcal M_\mathrm{FG}$. Since the quasi-coherent sheaf $\omega_{\M}\in\QCoh(\M)$ is obtained by pullback along the canonical map $\M\to\mathcal M_\mathrm{FG}$, it follows that it is also flat.
\end{ex}

\begin{lemma}\label{flatness and connective cover}
Let $X$ ba a  non-connective spectral stack, and let $c : X\to \tau_{\ge 0}(X)$ be the canonical map to its connective cover. The functor $c^* : \QCoh(\tau_{\ge 0}(X))\to\QCoh(X)$ restrict to an equivalence of symmetric monoidal $\infty$-categories $\QCoh^\flat(\tau_{\ge 0}(X))\simeq \QCoh^\flat(X)$, with inverse given by $\sF\mapsto\tau_{\ge 0}(\sF)$.
\end{lemma}

\begin{proof}
On the level of presheaves (and indeed, both are invariant under sheafification), both functors $\QCoh, \QCoh^\flat : \Fun(\CAlg, \mS)^\mathrm{op}\to \CAlg(\Cat)$ are defined as right Kan extension from the respective functors $\Mod, \Mod^\flat :\CAlg\to \CAlg(\Cat)$. The connective cover functor for spectral stack is also obtained by (sheafifying) the left Kan extension of the connective cover functor $\tau_{\ge 0} : \CAlg\to\CAlg^\mathrm{cn}$ It therefore suffices to exhibit an equivalence between the functor $\Mod^\flat : \CAlg\to\CAlg(\Cat)$ and the composite
$$
\CAlg\xrightarrow{\tau_{\ge 0}} \CAlg^\mathrm{cn}\xrightarrow{\Mod^\flat}\CAlg(\Cat).
$$
The canonical map $\tau_{\ge 0}(A)\to A$, natural in the $\E$-ring $A$, indeed induces a homotopy between these two functors, and to verify that it is an equivalence, it suffices to check it object-wise. We are reduced to showing that $A\otimes_{\tau_{\ge 0}(A)} -: \Mod_{\tau_{\ge 0}(A)}^\flat\to\Mod_A^\flat$ is an equivalence of $\infty$-categories with inverse $\tau_{\ge 0} : \Mod_A^\flat\to\Mod_{\tau_{\ge 0}(A)}^\flat$, which is established in \cite[Proof of Proposition 7.2.2.16, (3)]{HA}.
\end{proof}

\begin{prop}\label{higher connective covers}
The map $\beta: \Sigma^2(\omega)\to \sO_{\tau_{\ge 0}(\M)}$  induces an equivalence
$$
\tau_{\ge 2n}(\sO_{\M})\simeq \Sigma^{2n}(\omega^{\o n})
$$
in the $\i$-category $\QCoh(\tau_{\ge 0}(\M))$ for every $n\in \Z$.
\end{prop}

\begin{proof}
Similarly to the proof of Lemma \ref{connective cover from beta}, we can use the equivalence
$$
\Sigma^{2n}(\omega_{\M}^{\o n})\simeq (\Sigma^2( \omega_{\M}))^{\o n}\simeq \sO_{\M}^{\o n}\simeq \sO_{\M}
$$ 
of quasi-coherent sheaves on $\M$, to obtain the equivalence
$$
\tau_{\ge 2n}(\sO_{\M})\simeq \tau_{\ge 2n}(\Sigma^{2n}(\omega^{\o n}_{\M}))\simeq \Sigma^{2n}\tau_{\ge 0}(\omega_{\M}^{\o n})\simeq \Sigma^{2n}(\omega^{\o n})
$$
of quasi-coherent sheaves on $\tau_{\ge 0}(\M)$. The final equivalence here follows from the fact that the quasi-coherent sheaf $\omega_{\M}$ on $\M$ is flat, see Example \ref{example}, and that the functor $\tau_{\ge 0}:\QCoh^\flat(\M)\to\QCoh^\flat(\tau_{\ge 0}(\M))$ is symmetric monoidal by Lemma \ref{flatness and connective cover}.
\end{proof}

\subsection{Ind-coherent sheaves on spectral stacks over $\mathcal M_\mathrm{FG}$}\label{Section 1.5}
As in \cite[Subsection 2.1]{ChromaticCartoon}, we let $\mathcal M_\mathrm{FG}$ denote the spectral stack of formal groups in spectral algebraic geometry. We instead use $\Mo$ to denote its underlying ordinary stack, which by recovers the classical stack of formal groups in usual algebraic geometry. 

The following notion of ind-coherent sheaves does not coincide with the more general notion in derived algebraic geometry, e.g. \cite{GaRo}. It does not possess all the excellent formal properties of the latter (however, see Remark \ref{similarities with GaRo}). On the contrary, we define and use it only in a very specific and limited context, for which it is perfectly suited. This conforms to the usage of the term in \cite{Bartel Heard Valenzuela} and \cite{BHS}.

\begin{definition}\label{Def of IndCoh}
Let $f:X\to\mathcal M_\mathrm{FG}$ be a map of non-connective spectral stacks. The symmetric monoidal $\i$-category  of \textit{ind-coherent sheaves on $X$} is defined to be
$$
\IndCoh(X):=\Ind(\mD),
$$
where $\mD\subseteq\QCoh(X)$ is the thick subcategory (i.e.\ subcategory closed under finite limits, finite colimits, and extensions) spanned by the collection of quasi-coherent sheaves $f^*(\omega^{\o n}_{\mathcal M_\mathrm{FG}})$ for all $n\in \Z$.
\end{definition}

The inclusion $\mD\subseteq\QCoh(X)$ gives rise to a canonical functor $\IndCoh(X)\to\QCoh(X)$.
 The functor of \textit{global sections of an ind-coherent sheaf} $\sF\in\IndCoh(X)$ is taken to be composite
$$
\IndCoh(X)\to\QCoh(X)\xrightarrow{\Gamma(X; -)} \Sp.
$$
 It is clear that any map $g :X\to Y$ of spectral stacks over $\mathcal M_\mathrm{FG}$ induces a symmetric monoidal functor $g^* : \IndCoh(Y)\to\IndCoh(X)$, such that the diagram of symmetric monoidal $\i$-categories
 $$
\begin{tikzcd}
\IndCoh(Y) \ar{d}\ar{r}{g^*} & \IndCoh(X)\ar{d}\\
\QCoh(Y)\ar{r}{g^*} & \QCoh(X)
\end{tikzcd}
$$
commutes.

\begin{remark}
For $X =\M$ and $X=\Mo$, with the usual forgetful maps to $\mathcal M_\mathrm{FG}$, this recovers the $\i$-categories $\IndCoh(\M)$ of \cite[Definition 2.4.2]{ChromaticCartoon} and $\IndCoh(\Mo)$ of \cite[Definition 5.14]{BHS} respectively.
\end{remark}

\begin{remark}\label{Grothendieck prestable}
In the setting of Definition \ref{Def of IndCoh}, consider the full subcategory $\mD'\subseteq\QCoh(X)$, generated under finite colimits and extensions from the collection $f^*(\omega^{\o n}_{\mathcal M_\mathrm{FG}})$ for  $n\in \Z$. Clearly $\mD'\subseteq \mD$, and so the ind-completion $\IndCoh(X)_{\ge 0} := \Ind(\mD')$ admits a canonical fully faithful embedding into $\IndCoh(X)$. Its essential image may be identified with the full subcategory, spanned under colimits and extensions by the collection of ind-coherent sheaves $f^*(\omega^{\o n}_{\mathcal M_\mathrm{FG}})$ for  $n\in \Z$. It follows by \cite[Proposition 1.4.4.11]{HA} that there exists a $t$-structure on the stable $\i$-category $\IndCoh(X)$ for which $\IndCoh(X)_{\ge 0}$ is precisely the subcategory of connective objects. It follows from \cite[Proposition C.1.2.9, Example C.1.4.4]{SAG} that $\IndCoh(X)_{\ge 0}$ is a Grothendieck prestable $\i$-category in the sense of \cite[Definition C.1.4.2]{SAG} (for the analogous statement regarding quasi-coherent sheaves, see \cite[Proposition 9.1.3.1]{SAG}).
 If $X$ is a spectral stack (i.e.\, connective), then $\QCoh(X)$ also admits a canonical $t$-structure, and the functor $\IndCoh(X)\to\QCoh(X)$ is right $t$-exact; that is to say, it restricts to a colimit-preserving functor of Grothendieck prestable $\i$-categories $\IndCoh(X)_{\ge 0}\to \QCoh(X)^\mathrm{cn}$.
\end{remark}

We claim that $\omega\in \QCoh(\tau_{\ge 0}(\M))$ is of the required form for Definition \ref{Def of IndCoh} to apply. 

\begin{lemma}
There is a canonical map $f:\tau_{\ge 0}(\M)\to \mathcal M_\mathrm{FG}$ such that $\omega \simeq f^*(\omega_{\mathcal M_\mathrm{FG}}).$
\end{lemma}

\begin{proof}
The map $u : \M\to \mathcal M_\mathrm{FG}$, corresponding to ``forgetting the orientation'' of an oriented formal group, gives rise via the universal property of connective covers to a map $f:\tau_{\ge 0}(\M)\to \M$ for which the diagram of non-connective spectral stacks
$$
\begin{tikzcd}
\M \ar{r}{c}\ar{dr}[swap]{u}  & \tau_{\ge 0}(\M)\ar{d}{f} \\
  & \mathcal M_\mathrm{FG}
\end{tikzcd}
$$
commutes. That gives rise to equivalences
$$
c^*(f^*(\omega_{\mathcal M_\mathrm{FG}}))\simeq u^*(\omega_{\mathcal M_\mathrm{FG}})\simeq \omega_{\M}
$$
in $\QCoh^\flat(\M)$, and finally applying Lemma \ref{flatness and connective cover}, to $f^*(\omega_{\mathcal M_\mathrm{FG}})\simeq \tau_{\ge 0}(\omega_{\M}) = \omega$.
\end{proof}

\begin{cons}\label{construction of adjunction}
The map $\beta :\Sigma^2(\omega)\to\sO_{\tau_{\ge 0}(\M))}$ of Lemma \ref{connective cover from beta} induces a diagram
$$
\cdots \to \Sigma^{2(n+1)}(\omega^{\o n})\xrightarrow{\beta}\Sigma^{2n}(\omega^{\o n})\xrightarrow{\beta}\Sigma^{2(n-1)}(\omega^{\o n-1})\to\cdots
$$
of quasi-coherent sheaves on $\tau_{\ge 0}(\M)$. In fact, each of these quasi-coherent sheaves belongs to the thick subcategory $\mD\subseteq\QCoh(\tau_{\ge 0}(\M))$, thus we are looking at a functor  $\Omega^{2*}(\omega^{\o *}) : \mathbf Z_\mathrm{poset}^\mathrm{op}\to\mD.$ This functor is clearly symmetric monoidal, hence we may extend it to a symmetric monoidal functor $\Spfil\to\IndCoh(\tau_{\ge 0}(\M))$. 
\end{cons}

\begin{lemma}\label{existence of right adjoint}
The symmetric monoidal functor from Construction \ref{construction of adjunction} admits a colimit-preserving right adjoint.
\end{lemma}

\begin{proof}
We show how to construct such a functor $\IndCoh(\tau_{\ge 0}(\M))\to\Spfil$.
First  compose
$$
\mathbf Z_\mathrm{poset}^\mathrm{op}\times \mD\xrightarrow{\Sigma^{2*}(\omega^{\o *})\times \id} \mD\times \mD\subseteq \QCoh(\tau_{\ge 0}(\M))\times \QCoh(\tau_{\ge 0}(\M))
$$
with the quasi-coherent smash product and global sections functor
$$
\QCoh(\tau_{\ge 0}(\M))\times \QCoh(\tau_{\ge 0}(\M))\xrightarrow{-\o_{\sO} -}\QCoh(\tau_{\ge 0}(\M))\xrightarrow{\Gamma(\tau_{\ge 0}(\M); -)}\Sp.
$$
Together this gives a functor $\mathbf Z^{\mathrm{op}}_\mathrm{poset}\times \mD\to \Sp$, or by adjunction $\mD\to \Spfil$, and this functor is clearly exact. The desired colimit-preserving functor $\mathrm{IndCoh}(\tau_{\ge 0}(\M))\to \Spfil$ is thus obtained by the universal property of Ind-objects. Since global sections are right adjoint to the constant quasi-coherent sheaf functor, the adjointness claim follows from the construction of both of the functors involved.
\end{proof}

\begin{remark}\label{similarities with GaRo}
The continuity result of Lemma \ref{existence of right adjoint}, which will be used in the proof of Proposition \ref{IndCoh as filtered module}, highlights the importance of working with ind-coherent sheaves. Indeed, the construction of the right adjoint in question involves the functor of global sections $\Gamma(X; -):\QCoh(X)\to \Sp$, which does in general not preserve colimits. On the other hand, the ind-coherent global sections functor $\Gamma(X; -):\IndCoh(X)\to \Sp$ was defined in suich a way as to make it transparently colimit-preserving. Despite the difference between the two, this is a similarity between out notion of ind-coherent sheaves and that of \cite{GaRo}. One key property of the latter is that any map of derived stacks $f:X\to Y$ induces an adjunction $f_*:\IndCoh(X)\rightleftarrows \IndCoh(Y): f^!$, as opposed to the quasi-coherent sheaf adjunction $f^*:\QCoh(Y)\rightleftarrows \QCoh(X): f_*$. For $p:X\to *$ the terminal map, that recovers continuity of ind-coherent global sections $p_* = \Gamma(X; -)$, analogous to our setting.
\end{remark}

\begin{remark}
 The right adjoint functor $\IndCoh(\tau_{\ge 0}(\M))\to\Spfil$ of Lemma \ref{existence of right adjoint} is 
given object-wise by
$$
\sF\mapsto \Gamma(\tau_{\ge 0}(\M); \, \Sigma^{2*}(\omega^{\o *})\otimes_{\sO}\sF).
$$
Its symmetric monoidal left adjoint $\Spfil\to\IndCoh(\tau_{\ge 0}(\M))$ is a little harder to describe explicitly in general. But its value on the generators
$$
\mathbf 1(n)= ( \cdots \to 0\to 0\to S\xrightarrow{\id} S\xrightarrow{\id} S\to\cdots),
$$
with the first copy of $S$ appearing in filtered degree $n$, are not to hard to understand. It sends $\mathbf 1(n)\mapsto \Sigma^{2n}(\omega^{\o n})$.
\end{remark}

\begin{prop}\label{IndCoh as filtered module}
The adjunction $\Spfil\rightleftarrows\IndCoh(\tau_{\ge 0}(\M))$ of Construction \ref{construction of adjunction} and Lemma \ref{existence of right adjoint} is monadic, and exhibits a symmetric monoidal equivalence of $\i$-categories
$$
\IndCoh(\tau_{\ge 0}(\M))\simeq \Mod_{\Gamma (\tau_{\ge 0}(\M);\, \Sigma^{2*}(\omega^{\o *}))}(\Spfil).
$$
\end{prop}

\begin{proof}
This is a direct application of \cite[Proposition A.4]{BHS}. Indeed, the symmetric monoidal $\i$-category $\Sp^\mathrm{fil}$  is rigidly generated by the filtered spectra $\mathbf 1(n)$, the monoidal unit $\sO_{\tau_{\ge 0}(\M)}$ in $\IndCoh(\tau_{\ge 0}(\M))$ is compact, the right adjoint preserves colimits, and the essential image of the left adjoint contains a family of generators for $\IndCoh(\tau_{\ge 0}(\M))$, namely $\Sigma^{2n}(\omega^{\o n})$ for all $n\in\Z$.
\end{proof}

\begin{remark}
The above result is actually a special case of the ``deformation construction'' \cite[Proposition C.19]{BHS}. Only verification of the third point \cite[Definition C.13]{BHS} is not completely obvious, but may be handled by an application of Lemma \ref{flatness and connective cover}. Indeed, it requires us to verify that there is a  canonical homotopy equivalence
$$
\Map_{\IndCoh(\tau_{\ge 0}(\M))}(\Sigma^{2a}(\omega^{\o a}), \Sigma^{2b}(\omega^{\o b}))\simeq \Omega^\infty(S),
$$
induced by passage to global sections, for all $a>b$. To see that, we begin by noting that
$$
\Map_{\IndCoh(\tau_{\ge 0}(\M))}(\Sigma^{2a}(\omega^{\o a}), \Sigma^{2b}(\omega^{\o b})) \simeq
\Omega^{2(a-b)}\Map_{\IndCoh(\tau_{\ge 0}(\M))}(\sO, \omega^{\o (b-a)}),
$$
via standard limit-colimit preservation properties of Hom functors. Note that this already used the assumption that $a>b$. The remainder of the proof is to follow for any $n\ge 0$ the chain of canonical homotopy equivalences
\begin{eqnarray*}
\Omega^{2n}\Map_{\IndCoh(\tau_{\ge 0}(\M))}(\sO_{\tau_{\ge 0}(\M)}, \omega^{\o -n})
&\simeq&
\Omega^{2n}\Map_{\QCoh^\flat(\tau_{\ge 0}(\M))}(\sO_{\tau_{\ge 0}(\M)}, \omega^{\o -n})\\
&\simeq&
\Omega^{2n}\Map_{\QCoh^\flat(\M)}(\sO_{\M}, \omega_{\M}^{\o -n})\\
&\simeq&
\Omega^{2n}\Map_{\QCoh^\flat(\M)}(\sO_{\M}, \Sigma^{2n}(\sO_{\M}))\\
&\simeq&
\Map_{\QCoh^\flat(\M)}(\sO_{\M}, \Omega^{2n}\Sigma^{2n}(\sO_{\M}))\\
&\simeq&
\Omega^\infty(\sO(\M))\\
&\simeq &
\Omega^\infty (S).
\end{eqnarray*}
Here the first equivalence follows from the definition ind-coherent sheaves, the second equivalence follows from Lemma \ref{flatness and connective cover}, the third  equivalence is a consequence of the equivalence $\omega_{\M}\simeq \Sigma^{-2}(\sO_{\M})$ - a hallmark of oriented formal groups, the fourth equivalence follows from the standard exactness properties of Hom functors, the fifth equivalence is a result of the equivalence $\Omega^{2n}\Sigma^{2n}\simeq \mathrm{id}$, valid in the setting of any stable $\i$-category, and the final equivalence is \cite[Proposition 2.4.1]{ChromaticCartoon}.
\end{remark}

\subsection{Comparison to motivic spectra}\label{section main}
To obtain the main result we are after, it takes little but to combine everything we proved so far. We once again work with everything implicitly $p$-completed.

\begin{theorem}\label{main theorem}\label{section 6}
There is a canonical equivalence of symmetric monoidal $\i$-categories
$$
 \Sp^\mathrm{cell}_{\C}\simeq \IndCoh(\tau_{\ge 0}(\M)),
$$
compatible with the respective forgetful functors to $\Sp^\mathrm{fil}$ coming from Theorem \ref{Theorem Ctau} and Proposition \ref{IndCoh as filtered module}.
\end{theorem}

\begin{proof}
In light of Theorem \ref{Theorem Ctau} and Proposition \ref{IndCoh as filtered module}, both symmetric monoidal $\i$-categories in question are identified with filtered module $\i$-categories. Hence it suffices to identify the corresponding $\E$-algebra objects in $\Sp^\mathrm{fil}$. For that, note first that Proposition \ref{higher connective covers} induces an equivalence of filtered spectra
$$
\Gamma(\tau_{\ge 0}(\M);\, \Sigma^{2*}(\omega^{\o *}))\simeq \Gamma\big(\tau_{\ge 0}(\M);\, \tau_{\ge 2*}(\sO_{\M})\big),
$$
compatible with the $\E$-algebra structure.
Now, using the compatible simplicial presentations $\M\simeq \left|\Spec(\MP^{\otimes \bullet +1})\right|$ and $\tau_{\ge 0}(\M)\simeq \left|\Spec(\tau_{\ge 0}(\MP^{\otimes \bullet +1}))\right|$, we may identify
$$
\Gamma\big(\tau_{\ge 0}(\M);\, \tau_{\ge 2*}(\sO_{\M})\big)\simeq \Tot(\tau_{\ge 2*}(\MP^{\otimes \bullet +1})),
$$
which is precisely the $\E$-algebra in filtered spectra $\Omega^{0, *}_\mathrm{st}(S_{\C})$ by Lemma \ref{comparison lemma}.
\end{proof}

\begin{corollary}\label{dictionary}
Under the equivalence of Theorem \ref{main theorem}, we have the correspondence
\begin{eqnarray*}
S_{\C} &\leftrightarrow & \sO_{\tau_{\ge 0}(\M)}\\
S^{2,1} &\leftrightarrow & \omega^{\o-1}\\
S^{t,w} &\leftrightarrow&  \Sigma^{t-2w} (\omega^{\o-w})\simeq \Sigma^t({\tau}_{\ge -2w}(\sO_{\M})) \\
S^{0, -1}\xrightarrow{\tau} S_\C &\leftrightarrow & \Sigma^2(\omega)\xrightarrow{\beta}\sO_{\tau_{\ge 0}(\M)}\\
 S_\C/\tau &\leftrightarrow &  \sO_{\Mo}.
\end{eqnarray*}
In particular, we obtain an algebro-geometric formula for the bigraded homotopy groups of the motivic sphere spectrum, as quasi-coherent sheaf cohomology
$$
\pi_{t, w}(S_{\C})\simeq \mathrm H^{2w-t} (\tau_{\ge 0}(\M);\, \omega^{\o w}).
$$
\end{corollary}

The degeneration/deformation phenomena, as studied for instance in \cite{Gheorghe}, \cite{GWX}, \cite{GIKR}, are recovered easily from Theorem \ref{main theorem}. 

\begin{itemize}
\item
\underline{\textit{The generic fiber $\tau^{-1}$}:} To invert the $\tau$ in $\Sp_{\C}^\mathrm{cell}$, we should consider $\tau$-local objects. Under the dictionary of Corollary \ref{dictionary}, this corresponds to $\beta$-local objects in $\IndCoh(\tau_{\ge 0}(\M))$. But of course, inverting $\beta$ amounts to producing $\M$ out of $\tau_{\ge 0}(\M)$. And since maps of non-connective spectral stacks over $\mathcal M_\mathrm{FG}$, such as $\M\to\tau_{\ge 0}(\M)$, induce functors on ind-coherent sheaves, we obtain a chain of equivalences of $\i$-categories
$$
(\Sp^\mathrm{cell}_\C )^{\mathrm{Loc}(\tau)}\simeq \IndCoh(\tau_{\ge 0}(\M))^{\mathrm{Loc}(\beta)}\simeq \IndCoh(\M)\simeq \Sp.
$$
Only the last one does not follow from the previous discussion, and is instead the conclusion of \cite[Theorem 2.4.4]{ChromaticCartoon}.

\item
\underline{\textit{The special fiber $\tau = 0$} :} To set $\tau = 0$ in $\Sp^\mathrm{cell}_\C$, we shall consider modules over the quotient $S_{\mathbf C}/\tau$. Under the dictionary of Corollary \ref{dictionary}, this corresponds to modules over $\sO_{\tau_{\ge 0}(\M)}/\beta$ in $\IndCoh(\tau_{\ge 0}(\M))$. From the cofiber sequence of Lemma \ref{connective cover from beta}, we see that $\sO_{\tau_{\ge 0}(\M)}/\beta\simeq \sO_{\Mo}$. Hence the functoriality of ind-coherent sheaves along the canonical map of spectral stacks $\Mo\to\tau_{\ge 0}(\M)$ gives rise to the last in the chain of equivalences of $\i$-categories
$$
\Mod_{S_{\mathbf C}/\tau}(\Sp^\mathrm{cell}_\C )\simeq \Mod_{\sO_{\tau_{\ge 0}(\M)}/\beta}(\IndCoh(\tau_{\ge 0}(\M))\simeq \IndCoh(\Mo).
$$
\end{itemize}

\begin{remark}\label{Chow ts}
The $t$-structure on $\IndCoh(\tau_{\ge 0}(\M))$ from Remark \ref{Grothendieck prestable} coincides, under the equivalence of $\i$-categories of Theorem \ref{main theorem}, with the \textit{Chow $t$-structure} on cellular motivic spectra $\Sp_{\C}^\mathrm{cell}$ of \cite[Definition 2.25]{Chow t}. This is the $t$-structure determined through \cite[Proposition 1.4.4.11]{HA} by the subcategory $(\Sp_{\C}^\mathrm{cell})_{c\,\ge 0} \subseteq\Sp_\C^\mathrm{cell}$, which is spanned under extensions and colimits by the motivic sphere spectra $S^{t,w}$ with Chow(-Novikov) degree $c=t-2w \ge 0$. The equivalence of the latter with $\IndCoh(\tau_{\ge 0}(\M))_{\ge 0}$ in the sense of Remark \ref{Grothendieck prestable} now follows directly from Corollary \ref{dictionary}. The Chow $t$-structure is the essential instrument for approaching the ``cofiber of $\tau$ philosophy" in \cite{GWX}, and is the basis for its far-reaching integral (i.e.~non-$p$-complete) generalization over an arbitrary field (instead of $\C$) of \cite{Chow t}.
\end{remark}

\subsection{Comparison to synthetic spectra}\label{section synthetic}

Let us now drop the implicit $p$-completeness assumption of the previous section. Without it, the distinguished element $\tau\in\pi_{0, -1}(S_\C)$ no longer exists, making motivic spectra inaccessible through filtered spectra. Nonetheless, we may still obtain an ``integral'' comparison to another algebraic model, used to study these phenomena: the synthetic spectra of \cite{Pstragowski}.

Let us recall the construction of the $\i$-category $\mathrm{Sym}^\mathrm{ev}_E$ of \textit{even synthetic spectra based on $E$} from \cite[Definition 5.10]{Pstragowski} for any ring spectrum $E$ which is even Adams in the sense of \cite[Definition 5.8]{Pstragowski} (which includes all the ring spectra of interest to us here). 

\begin{cons}\label{cons synth}
First we define $\Sp^\mathrm{fpe}_E\subseteq \Sp^\mathrm{fin}$ to be the full subcategory spanned by all finite spectra $X$, such that $E_*(X)$ is a finitely-generated projective $\pi_*(E)$-module, concentrated purely in even degrees. Then $\mathrm{Sym}^\mathrm{ev}_E\subseteq \mathrm{Fun}^\pi ((\Sp^\mathrm{fpe}_E)^\mathrm{op}, \Sp)$ is the full subcategory  of product-preserving functors $F : (\mathrm{Sp}^\mathrm{fpe}_E)^\mathrm{op}\to\Sp$, satisfying the following additional condition. For any map $f: X\to Y$ in $\Sp_E^\mathrm{fpe}$ which induces a surjection $E_*(X)\to E_*(Y)$, the canonical map
$
F(Y)\to \fib( F(X)\to F(\fib(f)))
$
is an equivalence of spectra.
\end{cons}

The $\i$-category of synthetic spectra, which carries a canonical smash product symmetric monoidal structure, has several excellent properties, see \cite{Pstragowski}, \cite{Pstragowski on GH}, and \cite{BHS:Manifolds}.

\begin{theorem}\label{synthetic}
There is a canonical equivalence of symmetric monoidal $\i$-categories
$$
\IndCoh(\tau_{\ge 0}(\M))\simeq \mathrm{Syn}_{\mathrm{MP}}^\mathrm{ev}
$$
with even synthetic spectra based on the periodic complex bordism spectrum $\mathrm{MP}$.
\end{theorem}

\begin{proof}
First we repeat the proof of Theorem \ref{main theorem} (which did not rely on the $p$-completeness assumption) to identify $\IndCoh(\tau_{\ge 0}(\M))$ with modules in $\Sp^\mathrm{fil}$ for the filtered $\E$-ring $\Tot(\tau_{\ge 2*}(\MP^{\o \bullet +1}))$. The desired result now follows from identifying $\mathrm{Sym}_{\MP}^\mathrm{ev}$ with the same $\i$-category of filtered modules spectra. That is the content of (an even analogue of) \cite[Proposition C.12]{BHS}, together with the observation that the sphere spectrum is $\MP$-nilpotent complete in the sense of \cite[Definition 9.16]{BHS:Manifolds}, and so $\mathbf 1^\wedge_\tau\simeq \mathbf 1$ holds in $\mathrm{Sym}^\mathrm{ev}_\mathrm{MP}$  by \cite[Proposition A.11]{BHS:Manifolds}.
\end{proof}

\begin{remark}\label{MU over MP}
We may replace  $\MP$ with the usual complex bordism spectrum $\mathrm{MU}$ in the statement of Theorem \ref{synthetic}. That is because, $\MP$ being a sum of even-degree suspensions of $\mathrm{MU}$, we have $\Sp_{\MP}^\mathrm{fpe}\simeq \Sp_\mathrm{MU}^\mathrm{fpe}$. The conditions that either of $\MP_*(X)\to \MP_*(Y)$ and $\mathrm{MU}_*(X)\to\mathrm{MU}_*(Y)$ is a surjection are likewise equivalent, hence it follows from the definition of synthetic spectra, recalled above as Construction \ref{cons synth}, that  there is a canonical equivalence of $\i$-categories $\mathrm{Syn}_{\MP}^\mathrm{ev}\simeq \mathrm{Syn}_\mathrm{MU}^\mathrm{ev}$.
\end{remark}

Repeating all the arguments from the previous sections, and the proof of Theorem \ref{synthetic}, but this time using the simplicial presentation
$$
\mathcal M^{\mathrm{or}, \le n}_\mathrm{FG}\simeq \left |\Spec(E_n^{\otimes \bullet +1})\right |
$$
of the moduli stack of oriented formal groups of height $\le n$, see \cite[Remark 3.2.7]{ChromaticFiltration}, we obtain (with a single extra ingredient of the Smash Product Theorem) the analogous result:

\begin{varianthm}\label{variant}
For any prime $p$ and height $1\le n < \infty$, there is a canonical equivalence of symmetric monoidal $\i$-categories
$$
\IndCoh(\tau_{\ge 0}(\mathcal M_\mathrm{FG}^{\mathrm{or}, \le n}))\simeq L_n{\mathrm{S}}\mathrm{yn}_{E_n}^\mathrm{ev}
$$
with $E_n$-local even synthetic spectra based on the Lubin-Tate spectrum $E_n$.
\end{varianthm}

\begin{proof}
By repeating everything so far with $E_n$ replacing $\MP$, and once again applying the even analogue of  \cite[Proposition C.12]{BHS}, we obtain a symmetric monoidal equivalence of $\i$-categories
$$
\IndCoh(\tau_{\ge 0}(\mathcal M_\mathrm{FG}^{\mathrm{or}, \le n}))\simeq \Mod_{\mathbf 1^\wedge_\tau}(\mathrm{Syn}^\mathrm{ev}_{E_n}).
$$
By \cite[Proposition A.11]{BHS:Manifolds}, we can identify $\mathbf 1^\wedge_\tau\simeq L_n\mathbf 1$ with the $L_n$-localization in synthetic spectra.  By the Smash Product Theorem of Hopkins-Ravenel, $L_n$ is a smashing localization, hence $L_n\Sp\simeq \Mod_{L_nS}$. Expressing  $L_n$-localization of a stable presentable symmetric monoidal $\i$-category $\mC$ via the Lurie tensor product on $\mathcal P\mathrm{r}^\mathrm{L}$, this allows us to identify 
$$
L_n\mC\simeq L_n\Sp\o \mC\simeq \Mod_{L_nS}\o \mC\simeq \Mod_{L_nS\otimes \mathbf 1}(\mC).
$$
In the case in question $\mC = \mathrm{Syn}_{E_n}^{\mathrm{ev}}$, this gives the desired result.
\end{proof}

This too admits a deformation picture, this time related to Goerss-Hopkins obstruction theory.
Here the generic fiber is the $E_n$-local stable category
$$
L_n\mathrm{Sp}\simeq \IndCoh(\mathcal M_\mathrm{FG}^{\mathrm{or}, \le n})\simeq\QCoh(\mathcal M_\mathrm{FG}^{\mathrm{or}, \le n}),
$$
(these equivalences of $\i$-categories are due to \cite[Theorem 3.3.1 and Remark 3.3.6]{ChromaticFiltration}), while the special fiber is the $\i$-category of sheaves $\IndCoh(\mathcal M^{\heart, \le n}_\mathrm{FG})\simeq \QCoh(\mathcal M^{\heart, \le n}_\mathrm{FG})$ on the ordinary moduli stack $\mathcal M^{\heart, \le n}_\mathrm{FG}$ of formal groups of height $\le n$.

\begin{remark}
Let us indicate a direct approach to comparing ind-coherent sheaves on $\tau_{\ge 0}(\M)$ and synthetic spectra, not relying on the results of \cite{BHS}, and consequently, on the intermediate comparison with filtered spectra. The first step is to  composite functor
$$
\Sp^\mathrm{fin}\xrightarrow{-\o\sO_{\M}}\QCoh(\M)\xrightarrow{\tau_{\ge 0}}\QCoh(\tau_{\ge 0}(\M)).
$$ 
The first map is fully faithful by \cite[Lemma 2.4.5]{ChromaticCartoon}. Recall that flatness of quasi-coherent sheaves satisfies flat descent, so it can be checked upon pullback along the cover $q:\Spec(\MP)\to\M$. For a finite spectrum $X$, the sheaf $X\o\sO_{\M}$ is therefore flat if and only if $q^*(X\o \sO_{\M})\simeq X\o \MP$ is. This in particular happens if $\MP_*(X)=\pi_*(\MP\o X)$ is a graded projective $\pi_*(\MP)$-module, i.e.\ if we have $X\in\mathrm{Sp}_{\MP}^\mathrm{fpe}$. In that case, the first functor in the above composite factors through the subcategory  of flat quasi-coherent sheaves $\QCoh^\flat(\M)\subseteq\QCoh(\M)$. Lemma \ref{flatness and connective cover} therefore show that we obtain a fully faithful functor
$$
\Sp^\mathrm{fpe}_\mathrm{MP}\xrightarrow{\tau_{\ge 0}(-\o \sO_{\M})}\QCoh^\flat(\tau_{\ge 0}(\M)).
$$
Furthermore,  this functor sends by Lemma \ref{higher connective covers} for any $n\in \Z$
$$
\Sigma^{2n}(S)\mapsto \tau_{\ge 0}(\Sigma^{2n}(\sO_{\M}))\simeq \Sigma^{2n}\tau_{\ge -2n}(\sO_{\M})\simeq \omega^{\o -n},
$$
verifying that its essential image of the functor in question is contained inside the thick subcategory $\mathcal D\subseteq\QCoh(\tau_{\ge 0}(\M))$ from  Definition \ref{Def of IndCoh}. It follows that $\IndCoh(\tau_{\ge 0}(\M))\simeq \Ind(\mathcal D)$ admits a full subcategory, equivalent to $\Sp^\mathrm{fin}$. In particular, this gives rise to a full subcategory of the Grothendieck prestable $\i$-category $\IndCoh(\tau_{\ge 0}(\M))_{\ge 0}$ of Remark \ref{Grothendieck prestable}, closed under finite coproducts and generating in the sense of \cite[Definition C.2.1.1]{SAG}. The Gabriel-Popescu Theorem \cite[Corollary C.2.1.10]{SAG} then gives rise to an adjunction
$$
\mathrm{Fun}^\pi((\mathrm{Sp}_\mathrm{MP}^\mathrm{fpe})^\mathrm{op}, \Sp)\rightleftarrows \IndCoh(\tau_{\ge 0}(\M))
$$
with a fully faithful right adjoint and a $t$-exact left adjoint. In particular, the right-hand side is some kind of reflective localization of the $\i$-category of the product-preserving functors $(\Sp^{\mathrm{fpe}}_{\MP})^{\mathrm{op}}\to\Sp$. Theorem \ref{synthetic} then reduces to verifying that this localization is the same one as the one induced by the condition on such spectrum-valued presheaves given in Construction \ref{cons synth}.
\end{remark}

\begin{remark}
In line with the description of the Chow $t$-structure from Remark \ref{Chow ts}, the  \textit{synthetic analogue} functor $\nu :\Sp\to \mathrm{Syn}_\mathrm{MP}^\mathrm{ev}$ of \cite[Definition 4.3]{Pstragowski} (really an even analogue thereof) may be expressed through the equivalence of $\i$-categories of Theorem \ref{synthetic}
as
$$
\Sp\simeq\IndCoh(\M)\simeq \IndCoh(\tau_{\ge 0}(\M))^{\mathrm{Loc}(\beta)}\xrightarrow{\tau_{\ge 0}}\IndCoh(\tau_{\ge 0}(\M)),
$$
where the first equivalence of $\i$-categories is the one from \cite[Theorem 2.1.4]{ChromaticCartoon}.
Compare this with the description of the analogous $\mathrm C_2$-equivariant version of the functor $\nu$ over the base-field $\mathbf R$ in \cite[Definition 6.17]{BHS}.
\end{remark}

\section{Periodic spectral stacks and the deformation picture}\label{Part 2}
We investigate the special properties of the spectral stacks $\M$ and $\tau_{\ge 0}(\M)$, that lead to the $\tau$-deformation picture.
Some of this section is inspired by the treatment of periodicity in \cite[Section 2]{Pstragowski on GH}.

\subsection{Non-connective basic affines}\label{subsection D}

In classical algebraic geometry, for a global function $f\in \sO(X)$ on a scheme (or stack) $X$, we have the basic affine $D_X(f)\subseteq X$, giving the canonical open subscheme (resp.~substack) structure to the open locus $\{f\ne 0\}$. In this subsection, we consider a generalization of this construction in two ways:
\begin{enumerate} 
\item We work in 
 non-connective spectral algebraic geometry as opposed to classical algebraic geometry. 
\item Whereas a global function $f$ on $X$ is equivalent to a map of quasi-coherent sheaves $f:\sO_X\to\sO_X$, we allow $f$ to have as its domain any invertible sheaf on $X$.
\end{enumerate}
\begin{definition}\label{f-locality}
Let $X$ be a  non-connective  spectral stack. Fix an invertible quasi-coherent sheaf $\mL$ on $X$ and map $f: \mL\to \sO_X$ in $\QCoh(X)$. A quasi-coherent sheaf $\sF$ on $X$ is \textit{$f$-local} if the map $f : \mL\o_{\sO_X}\sF\to \sF$ is an equivalence in the $\i$-category $\QCoh(X)$. We denote the full subcategory of $f$-local quasi-coherent sheaves by $\QCoh(X)^{\mathrm{Loc}(f)}\subseteq\QCoh(X)$.
\end{definition}

The full subcategory $\QCoh(X)^{\mathrm{Loc}(f)}\subseteq\QCoh(X)$ is easily seen to be closed under smash product of quasi-coherent sheaves, and as such inherits a symmetric monoidal structure.

\begin{prop}\label{f-localization}
The subcategory inclusion $\QCoh(X)^{\mathrm{Loc}(f)}\subseteq\QCoh(X)$ admits a symmetric monoidal left adjoint $\sF\mapsto \sF[f^{-1}]$, where $$\sF[f^{-1}]\simeq \varinjlim(\sF\xrightarrow{f}\mL^{\o -1}\o_{\sO_X}\sF\xrightarrow{f}\mL^{\o -2}\o_{\sO_X}\sF\to\cdots).$$
This exhibits $\QCoh(X)^{\mathrm{Loc}(f)}$ as a smashing localization of $\QCoh(X)$, which is to say that the canonical map $\sF\o_{\sO_X}\sO_{X}[f^{-1}]\to \sF[f^{-1}]$ is an equivalence for all $\sF\in\QCoh(X)$.
\end{prop}

\begin{corollary}\label{localization and modules}
There is a canonical equivalence of symmetric monoidal $\i$-categories
$$
\QCoh(X)^{\mathrm{Loc}(f)}\simeq \Mod_{\sO_X[f^{-1}]}(\QCoh(X)).
$$
\end{corollary}

\begin{proof}
We must show that the adjunction
$$
(-)[f^{-1}]:\QCoh(X)\rightleftarrows \QCoh(X)^{\mathrm{Loc}(f)}
$$
is monadic. By the Bar-Beck-Lurie Theorem, if suffices to show that the right adjoint preserves colimits.
Consider a diagram $\sF_i\in \QCoh(X)^{\mathrm{Loc}(f)}$. Composing with the inclusion $\QCoh(X)^{\mathrm{Loc}(f)}\subseteq \QCoh(X)$, let $\sF := \varinjlim_i \sF_i$ denote the colimit computed in $\QCoh(\mX)$. Since the $f$-localization functor, being a left adjoint, preserves colimits, it follows that
$$
\sF[f^{-1}]\simeq \varinjlim_i (\sF_i)[\beta^{-1}]\simeq \varinjlim_i \sF_i,
$$
where the colimit on the right-hand-side is computed in $\QCoh(X)^{\mathrm{Loc}(f)}$. It therefore remains to show that the canonical map $\sF\to \sF[f^{-1}]$ is an equivalence in $\QCoh(X)$. That follows from the smashing localization assertion  $\sF[f^{-1}]\simeq \sF\o_{\sO_X}\sO_X[f^{-1}]$ of Proposition \ref{f-localization}, and the fact that smashing with $\sO_X[f^{-1}]$ (or any fixed quasi-coherent sheaf) commutes with colimits as a functor $\QCoh(X)\to\QCoh(X)$.
\end{proof}

\begin{remark}
The proof of Corollary \ref{localization and modules} shows more generally that any smashing localization in a presentably symmetric monoidal $\i$-category is monadic.
\end{remark}

To express $f$-local quasi-coherent sheaves more geometrically, 
we need to  embark on a short digression regarding affine morphisms in non-connective spectral algebraic geometry.

\begin{cons}
By the usual yoga of representability, a map of non-connective spectral stacks $f:\mY\to \mX$ is defined to be \textit{affine} if and only if for any $\E$-ring $A$ and map $\Spec(A)\to \mX$, the pullback morphism $\mY\times_{\mX}\Spec(A)\to \Spec(A)$ is an affine map of non-connective spectral schemes. In particular, that implies that $\mY\times_\mX\Spec(A)\simeq \Spec(B)$ for some $\E$-$A$-algebra $B\in\CAlg_A$. It follows that the quasi-coherent pushforward $f_*(\sO_\mY)$ caries a canonical structure of an $\E$-algebra in $\QCoh(\mX)$. Conversely, if $\sA\in\CAlg(\QCoh(\mX))$, then we may associate to it an affine morphism of non-connective spectral stacks $\Spec_\mX(\sA)\to\mX$. The universal property of the latter may be given by
$$
\Map_{(\mathcal S\mathrm{hv}_\mathrm{fpqc}^\mathrm{nc})_{/X}}(\mY, \Spec_\mX(\sA))\simeq \Map_{\CAlg(\QCoh(\mX))}(\sA, f_*(\sO_{\mY}))
$$
for any relative non-connective spectral stack $f:\mY\to \mX$. If $\mathrm{Aff}^\mathrm{nc}(\mX)\subseteq(\mathcal S\mathrm{hv}^\mathrm{nc}_\mathrm{fpqc})_{/X}$ denotes the full subcategory of affine morphisms to $\mX$, then the constructions $(f:\mY\to \mX)\mapsto f_*(\sO_{\mY})$ and $\Spec_\mX$ described above induce an equivalence of $\i$-categories
$$
\mathrm{Aff}^\mathrm{nc}(\mX)\simeq \CAlg(\QCoh(\mX))^\mathrm{op},
$$
thanks to (a non-connective analogue of) \cite[Proposition  6.3.4.5.]{SAG}.
\end{cons}

\begin{remark}
Note that in the above sense, a connective spectral stack $X$ admits affine morphisms $\mY\to X$ from many non-connective spectral stacks $\mY$. In fact, the $\mY\simeq \Spec_X(\sA)$ is a (connective) spectral stack if and only if the $\E$-algebra $\sA$ is a connective quasi-coherent sheaf on $X$.
\end{remark}

\begin{remark}
It is clear that, for any affine map $\mY\to \mX$ into a geometric non-connective spectral stack $\mX$, the non-connective spectral stack $\mY$ is also geometric. 
\end{remark}

\begin{prop}\label{sheaves on affines}
For any affine map of geometric non-connective spectral stacks $f:\mY\to \mX$, the quasi-coherent sheaf adjunction
$$
f^*:\QCoh(\mX)\rightleftarrows \QCoh(\mY) : f_*
$$
induces an equivalence of symmetric monoidal $\i$-categories
$$
\QCoh(\mY)\simeq \Mod_{f_*(\sO_\mY)}(\QCoh(\mX)).
$$
\end{prop}

\begin{proof}
If $\mX$, and consequently also $\mY$, is affine, this is immediate. Thanks to the geometricity assumption, we may write the $\i$-categories of quasi-coherent sheaves in questions as compatible totalizations. The non-affine result then follows, repeating the proof of its connective analogue \cite[Proposition 6.3.4.6]{SAG}, form the affine one by passing to totalizations, thanks to the Beck-Chevalley compatibility of monadicity with totalizations \cite[Theorem 4.7.5.2]{HA}.
\end{proof}

In the setting of Definition \ref{f-locality}, it follows from Proposition \ref{f-localization} that $\sO_X[f^{-1}]$ carries a canonical structure of an $\E$-algebra object in $\QCoh(X)$. The \textit{basic affine over $X$, determined by $f$} is then defined to be the non-connective spectral stack
$$
D_X(f) := \Spec_X(\sO_X[f^{-1}]).
$$
By definition, it admits a canonical (non-connective) affine morphism $D_X(f)\to X$.

\begin{prop}\label{QC is loc}
Assume that the  non-connective spectral stack $X$ is geometric.
The affine map of non-connective spectral stacks $D_X(f)\to X$ exhibits an equivalence of symmetric monoidal $\i$-categories
$$
\QCoh(D_X(f))\simeq \QCoh(X)^{\mathrm{Loc}(f)}.
$$
\end{prop}

\begin{proof}
This follows from combining Corollary \ref{localization and modules} with Proposition \ref{sheaves on affines}.
\end{proof}

Returning to the classical setting of the start of Subsection \ref{subsection D}, a function $f\in \sO(X)$ in classical algebraic geometry does not give rise merely to the basic affine $D_X(f)$, but also to the closed subscheme $V_X(f)$, the vanishing locus $\{f = 0\}$, counted with multiplicity. It is given explicitly by $V_X(f)\simeq \Spec_X(\sO_X/f)$. Unfortunately, this construction can not be carried over to spectral algebraic geometry, non-connective or otherwise, in full generality. That is due to the quotients of $\E$-rings in general failing to inherit an $\E$-ring structure themselves; they may not even carry one, as is famously the case for Moore spectra $S/p$. Hence it is unclear how to make sense of $\sO_X/f$ in the spectral context. In the next subsection, we consider a special case in which this difficulty disappears.

\subsection{Effective Cartier divisors}

Recall,  e.g.\ from \cite[\href{https://stacks.math.columbia.edu/tag/01WQ}{Tag 01WQ}]{stacks-project}, the usual notion of effective Cartier divisors in ordinary algebraic geometry. An \textit{effective Cartier divisor} on an ordinary scheme $X$ is a closed immersion of ordinary schemes $D\to X$, such that the ideal sheaf $\sO_X (-D):=\mathrm{ker}(\sO_X\to\sO_D)$ is an invertible quasi-coherent sheaf on $X$. Most of these words have natural interpretations in spectral algebraic geometry as well. For instance:

\begin{definition}
A map of spectral stacks $Y\to X$ is a \textit{closed immersion} if it is affine, and the underlying map of ordinary stacks $Y^\heart\to X^\heart$ is a closed immersion in the sense of ordinary algebraic geometry.
\end{definition}

\begin{remark}
The condition that the map of underlying ordinary stacks is a closed immersion may be rephrased as follows. Since $Y\to X$ is a affine, it is fully determined by the $\E$-algebra object $\sO_Y\in\CAlg(\QCoh(X))$ (we are abusing notation by omitting to write a quasi-coherent push-forward). It being a closed immersion is then equivalent to the map of sheaves of commutative algebras $\pi_0(\sO_X)\to\pi_0(\sO_Y)$ being surjective.
\end{remark}

\begin{remark}
Since being a closed immersion is local for the flat topology, the definition of a flat immersion may yet equivalently be extended via descent from affines. There a map of connectives $\E$-ring $A\to B$ corresponds to a closed immersion $\Spec(B)\to\Spec(A)$ if and only if the ring map $\pi_0(A)\to\pi_0(B)$ is surjective.
\end{remark}

There is therefore nothing stopping us from defining an analogue of effective Cartier divisors in spectral algebraic geometry.

\begin{definition}\label{Def of effCart}
Let $X$ be a  spectral stack. An \textit{extended effective Cartier divisors on $X$} is a closed immersion of spectral stacks $D\to X$, such that $\sO_X(-D):=\mathrm{fib}(\sO_X\to\sO_D)$ is an invertible quasi-coherent sheaf on $X$.
\end{definition}

\begin{remark}\label{Remark 2.4.2}
Comparing Definition \ref{Def of effCart} to the usual notion of effective Cartier divisors in ordinary algebraic geometry, though the definitions are verbatim the same, the notion of an extended effective Cartier divisors is more general in a number of aspects.
\begin{enumerate}[label =(\arabic*)]
\item In spectral algebraic geometry, a closed immersion $D\to X$ into an ordinary scheme $X$ need not imply that the non-connective spectral stack $D$ itself is an ordinary scheme.
\item In our definition of $\sO_X(-D)$, we define it to be the $\i$-categorical fiber, a derived notion of the ordinary kernel.\label{point 2 of Remark 2.4.2}
\item The demand that $\sO_X(-D)$ be an invertible quasi-coherent sheaf is much less restrictive in spectral than its classical analogue is in ordinary algebraic geometry.\label{point 3 of Remark 2.4.2}
\end{enumerate}
\end{remark}

\begin{remark}
To expand upon point \ref{point 3 of Remark 2.4.2} or Remark \ref{Remark 2.4.2}, let us recall from  \cite[Section 2.9]{SAG} the distinction between invertible sheaves and line bundles in spectral algebraic geometry (connective or otherwise). A quasi-coherent sheaf on a (non-connective) spectral stack $X$ is \textit{invertible} if it is invertible with respect to the usual symmetric monoidal structure $\o_{\sO_X}$ on $\QCoh(X)$, and is a \textit{line bundle} if it is locally free of rank $1$, i.e.~locally isomorphic to $\sO_X$. Every line bundle is invertible, but not conversely. Indeed, an invertible sheaf may locally be of the form $\Sigma^n(\sO_X)$ for any $n\in \Z$. In accordance with the notation of \cite{SAG}, line bundles span the \textit{Picard $\i$-groupoid} $\mathscr P\mathrm{ic}(X)$, while invertible sheaves span the \textit{extended Picard $\i$-groupoid} $\mathscr P\mathrm{ic}^\dagger(X)$, of which only the first restricts to the usual notion from algebraic geometry upon restriction to an ordinary scheme $X$.
\end{remark}

The point \ref{point 2 of Remark 2.4.2} of Remark \ref{Remark 2.4.2} accounts for a way in which the spectral notion of extended effective Cartier divisors is better behaved than its classical counterpart: extended effective Cartier divisors pullback without any caveats.

\begin{prop}
Let $f:Y\to X$ be a map of spectral stacks. If $D\to X$ is an effective Cartier divisor, then so is $D\times_X Y\to Y$.
\end{prop}

\begin{proof}
Let us denote $D' :=  D\times_X Y$. Since closed immersions are closed under pullback, it is clear that $D'\to Y$ is a closed immersions. On the other hand, observing that
\begin{eqnarray*}
\sO_Y(D') &\simeq& \mathrm{fib}(\sO_{D'}\to \sO_Y)\\
&\simeq& \mathrm{fib}(f^*(\sO_D)\to f^*(\sO_X))\\
&\simeq& f^*(\mathrm{fib}(\sO_D\to \sO_X))\\
&\simeq& f^*(\sO_X(-D)),
\end{eqnarray*}
invertibility of this quasi-coherent sheaf follows from the fact that the pullback functor $f^*$, being symmetric monoidal, preserves invertibility.
\end{proof}

In light of the compatibility with pullbacks, extended
effective Cartier divisors  assemble into a functor $\mathrm{CDiv}^\dagger_\mathrm{eff}:\CAlgcn\to \mS$.  Since both the conditions of being a closed immersion and an invertible sheaf satisfy flat descent, the functor $\mathrm{Div}$ is a spectral stack - the \textit{moduli stack of extended effective Cartier divisors}.

\begin{remark} The moduli stack of extended effective Cartier divisors may be given explicitly by
$$
\mathrm{CDiv}^\dagger_\mathrm{eff}(A)\simeq (\mathrm{CAlg}_{\pi_0(A)}^{\heart, \mathrm{surj}})^\simeq \times_{\CAlg_{\pi_0(A)}^\simeq}(\CAlg^\mathrm{cn}_A)^\simeq\times_{\Mod_A^\simeq}\mathscr P\mathrm{ic}^\dagger(A),
$$
where the one non-evident functor  is $\CAlg^\mathrm{cn}_A\to\Mod_A$, sending a connective $\E$-$A$-algebra $B$ to the $A$-module $\mathrm{fib}(A\to B)$.
\end{remark}

\begin{remark}
Alternatively, we could define an \textit{effective Cartier divisor} on a  spectral stack $X$ the same as above, but replacing every instance of `invertible sheaf' with `line bundle'. This would lead to a stack $\mathrm{CDiv}_\mathrm{eff}$, somewhat closer to classical effective Cartier divisors of ordinary algebraic geometry. We will have no use for that notion here, however.
\end{remark}

Let $D\to X$ be a an extended effective Cartier divisor. It is an affine map of non-connective spectral stacks,  so the ``divisor itself'' is given by $D := \Spec_X(\sO_D)$.  If we denote the canonical map by $i:\sO_X(-D)\to\sO_X$, then the \textit{complement of the divisor} $X-D\to X$ is taken to be 
$$
X- D :=D_X(i) \simeq \Spec_X(\sO_X[i^{-1}]);
$$
a special case of the non-connective basic affines of Subsection \ref{subsection D}.

\begin{remark}
Though we view the complement of an extended effective Cartier divisor $X-D\to X$ as a form of a geometric complement to the closed immersion $D\to X$, it is not necessarily true that it would be an open immersion. That is because open immersion are affine-locally of the form $A\to A[f^{-1}]$ for $f\in\pi_0(A)$, while the complement of an extended effective Cartier divisor may be affine-locally of the form $A\to A[f^{-1}]$ for $f\in \pi_n(A)$ and any $n\in \Z$. On the other hand, having an empty intersection with the divisor $D$, in the sense that $(X-D)\times_XD\simeq \emptyset$, justifies calling $X-D$ a complement.
\end{remark}

Both of the affine maps $D\to X$ and $X-D\to X$, associated to an extended effective Cartier divisor $D$ in $X$, admit a universal case on the moduli stack $\mathrm{CDiv}^\dagger_\mathrm{eff}$.

\begin{cons}
A map of non-connective spectral stacks $X\to \mathrm{CDiv}^\dagger_\mathrm{eff}$ is by design equivalent to specifying an extended effective Cartier divisors $D$ on  $X$. In particular, there exists a \textit{universal effective Cartier divisor} $D_\mathrm{univ}\to \mathrm{CDiv}_\mathrm{eff}^\dagger$, and the \textit{universal  complement} $\mathrm{CDiv}_\mathrm{eff}^\dagger -D_\mathrm{univ}\to\mathrm{CDiv}_\mathrm{eff}^\dagger$.
 For any effective Cartier divisor $D\to X$, equivalent to a map $X\to \mathrm{Div}$, there are pullback square of non-connective spectral stacks
$$
\begin{tikzcd}
D \ar{r} \ar{d} & X\ar{d} \\
D_\mathrm{univ}\ar{r} & \mathrm{CDiv}_\mathrm{eff}^\dagger,
\end{tikzcd}
 \qquad\qquad
 \begin{tikzcd}
X-D \ar{r} \ar{d} & X\ar{d} \\
\mathrm{CDiv}_\mathrm{eff}^\dagger -D_\mathrm{univ}\ar{r} & \mathrm{CDiv}_\mathrm{eff}^\dagger.
\end{tikzcd}
$$
\end{cons}

We view the cospan of  spectral stacks
$$
\mathrm{CDiv}_\mathrm{eff}^\dagger - D_\mathrm{univ}
\to\mathrm{CDiv}_\mathrm{eff}^\dagger\leftarrow D_\mathrm{univ}
$$
as an analogue in spectral algebraic geometry of the `deformation cospan' of ordinary stacks
$$
\Spec(\mathbf Z)\simeq \mathbf G_m/\mathbf G_m\to \mathbf A^1/\mathbf G_m\leftarrow\mathrm B\mathbf G_m,
$$
induced upon passage to quotients under the scaling action of $\G_m$ from the  complementary open and closed immersions $\G_m\to \mathbf A^1\leftarrow\{0\}$; the generic and special fiber respectively. To see the reason behind the claimed analogy to extended effective Cartier divisors, we consider the analogous situation in ordinary algebraic geometry.

\begin{remark}\label{connection def}
Recall from \cite[Definition 12.3.2]{Olsson} that a \textit{generalized effective Cartier divisor} (also known as a \textit{Deligne ``divisor''}) on an ordinary stack $X$ consists of a pair $(\mL, s)$ of a line bundle $\mL$ on $X$ and a map of quasi-coherent sheaves $s:\mL\to\sO_X$. Any effective Cartier divisor $D\to X$ is a special case of a generalized one, via the map $\sO_X(-D)\to \sO_X$. Generalized effective Cartier divisors are by \cite[Proposition 10.3.7]{Olsson} classified by the moduli stack $\mathbf A^1/\mathbf G_m$, the quotient stack of the scaling action of the multiplicative group $\G_m$ on the affine line $\A^1$. 
For any effective Cartier divisor $D\to X$, we therefore obtain the map of ordinary stacks $X\to \mathbf A^1/\mathbf G_m$. It fits into the `deformation diagram' 
$$ 
\begin{tikzcd}
X- D\ar{r} \ar{d} &X\ar{d} & D \ar{l}\ar{d}\\
 \Spec(\Z)\ar{r} & \mathbf A^1/\G_m & \mathrm B\mathbf G_m\ar{l},
\end{tikzcd}
$$
in which 
 both squares are pullback squares in ordinary stacks. 
\end{remark}

\begin{remark}\label{nogolemma}
Some aspects of the generalized effective Cartier divisor story carry over with little change to the context of spectral algebraic geometry\footnote{A different, though closely related, analogue of generalized effective Cartier divisors in derived algebraic geometry was studied in \cite[Subsection 3.2]{Khan}.}. We must replace the ordinary stack $\mathrm B\mathbf G_m$, classifying line bundles on ordinary schemes, with the  spectral stack $\mathrm{BGL}_1^\dagger$ from  \cite[Construction 4.1.5]{Elliptic 1}, classifying invertible quasi-coherent sheaves. This is a delooping of the spectral group scheme $\GL_1$ of  \cite[Subsection 1.6.3]{Elliptic 2}, but not the usual connected delooping $\mathrm{BGL}_1$, which classifies line bundles instead. Despite thus not being a classifying stack of a group itself, we can still treat $\mathrm{BGL}_1^\dagger$ as close to it. For instance, we can define a  spectral stack $\mathbf A^1/\mathrm{GL}_1^\dagger$, classifying pairs $(\mL, s)$ of an invertible quasi-coherent sheaf $\mL$ and a map $s:\mL\to\sO_X$. This is a spectral analogue of the ordinary stack $\A^1/\G_m$, and even fits into a pullback square
$$
\begin{tikzcd}
\A^1\ar{r} \ar{d} & \Spec(S) \ar{d} \\
\textup{$\mathbf A^1/\GL_1^\dagger$} \ar{r} & \textup{$\mathrm{BGL}_1^\dagger$}.
\end{tikzcd}.
$$
Here $\A^1 = \Spec(S\{t\})$ is the smooth affine line over $\Spec(S)$, i.e.\ the affine spectral scheme corresponding to the free $\E$-ring $S\{t\}$.
Considering only those map $s:\sL\to \sO_X$ which are equivalences of quasi-coherent sheaves, produces a canonical map $\Spec(S)\to\A^1/\GL_1^\dagger$. Conversely, any invertible bundle $\sL$ may be associated the zero map $\sL\to 0\to\sO_X$, giving a canonical map $\mathrm{BGL}^\dagger_1\to\mathbf A^1/\GL_1^\dagger$.
\end{remark}

\begin{remark}\label{Remark problems 1}
Continuing with setting of the previous Remark,  there is, as in the classical case, a canonical map $\mathrm{CDiv}_\mathrm{eff}^\dagger\to\mathbf A^1/\mathrm{GL}_1^\dagger$. For any spectral stack $X$, it sends an extended effective Cartier divisor $D\to X$ to the map $\sO_X(-D)\to \sO_X$ in $\QCoh(X)$. This induces a commutative diagram of non-connective spectral stacks
$$
\begin{tikzcd}
X- D\ar{r} \ar{d} &X\ar{d} & D \ar{l}\ar{d}\\
\textup{$\mathrm{CDiv}^\dagger_\mathrm{eff}$} - D_\mathrm{univ}  \ar{r} \ar{d} & \textup{$\mathrm{CDiv}_\mathrm{eff}^\dagger$}\ar{d} & D_\mathrm{univ} \ar{l}\ar{d}\\
 \Spec(S)\ar{r} & \textup{$\mathbf A^1/\GL_1^\dagger$} & \textup{$\mathrm{BGL}_1^\dagger$}\ar{l},
\end{tikzcd}
$$
in which the upper two inner squares  are pullback squares, so is the lower left-hand inner square, and consequently so is the big left-hand square. 
The point of departure from the classical story, recounted in Remark \ref{connection def} is that the right-hand square, and consequently the big right-hand one,  is not Cartesian. Indeed, the pullback $X\times_{\mathbf A^1/\mathrm{GL}_1^\dagger}\mathrm{BGL}_1^\dagger\to X$ is a closed immersion into $X$, corresponding to the $\mathbb E_\infty$-algebra quotient of the map $\sO_X(-D)\to \sO_X$. That is to say, it corresponds to the quasi-coherent $\mathbb E_\infty$-$\sO_X$-algebra $\mathrm{Sym}_{\sO_X}^*(0)\o_{\Sym^*_{\sO_X}(\sO_X(-D))}\sO_X.$ On the other hand, the closed immersion $D\to X$ corresponds to the quasi-coherent $\E$-$\sO_X$-algebra $\sO_D\simeq \mathrm{cofib}(\sO_X(-D)\to \sO_X)$, with the cofiber computed in the $\i$-category $\QCoh(X)$. Unlike if we were performing these two quotienting operation with simplicial commutative algebras, i.e.\ in the world of derived algebraic geometry, they in general do not coincide in the present $\E$-ring setting.
 Therefore, in the spectral context, the `deformation picture' genuinely needs to be phrased in terms of extended effective Cartier divisors, instead of in terms of $\mathbf A^1/\GL_1^\dagger$.
\end{remark}

\subsection{Periodic spectral stacks}
In the world of connective spectral algebraic geometry, and spectral stack $X$ admits a canonical map $X^\heart\to X$ from its underlying ordinary stack $X^\heart$. Since this map is always affine by \cite[Corollary 9.1.6.7]{SAG}, and induces an isomorphism on underlying ordinary stacks, it is a closed immersion. It is natural to ask when it is an effective divisor.

\begin{definition}
A geometric spectral stack $X$ is \textit{periodic} if the canonical map from the underlying ordinary stack  $X^\heart\to X$ is an extended effective Cartier divisor.
\end{definition}

If $X^\heart\to X$ is an extended effective Cartier divisor, its fundamental fiber sequence is the (co)fiber sequence
$$
\tau_{\ge 1}(\sO_X)\to \sO_X\to \pi_0(\sO_X)\simeq \sO_{X^\heart}
$$
in $\QCoh(X)$.
It follows that a geometric spectral stack  $X$ is  periodic if and only if its $1$-connective cvoer $\tau_{\ge 1}(\sO_X)$ is an invertible quasi-coherent sheaf on $X$.

\begin{remark}
A large class of examples of periodic spectral stacks are therefore \textit{$n$-periodic spectral stacks}, i.e.\ such geometric spectral stacks $X$ for which $\tau_{\ge 1}(\sO_X)\simeq \Sigma^n(\mL)$ for $\mL$ a line bundle on $X$. The $1$-connective cover $\tau_{\ge 1}(\sO_X)$ is by definition $1$-connective. Since any line bundle $\sL$ is always flat and hence connective (because we are assuming the same for $X$), it follows that an equivalence of quasi-coherent sheaves $\tau_{\ge 1}(\sO_X)\simeq \Sigma^n(\sL)$ is possible only for $n\ge 1$. A (non-empty) $n$-periodic stack will therefore always have period $n\ge 1$. Conversely, each connected component of a periodic spectral Deligne-Mumford stack is $1$-periodic for some $n\ge 0$ by \cite[Remark 2.9.5.8]{SAG}. In particular, any periodic spectral stack is fpqc-locally $n$-periodic for some (collection of) positive integers $n\ge 1$.
\end{remark}

For $X$ a periodic spectral stack, let us denote by $\mX := X- X^\heart$ the complement of the extended effective Cartier divisor $X^\heart\to X$. Equivalently, $\mX\simeq D_X(\beta)$ is the basic open associated to the quasi-coherent sheaf map $\beta : \tau_{\ge 1}(\sO_X)\to \sO_X$. In particular, we have an identification $\mX\simeq\Spec_X(\sO_X[\beta^{-1}])$, and as such there is a canonical affine map of non-connective spectral stacks $\mX\to X$.

\begin{lemma}\label{periodic homotopy inverted}
The map $\mX \to X$ exhibits an equivalence of spectral stacks $X\simeq \tau_{\ge 0}(\mX)$.
\end{lemma}

\begin{proof}
We must show that the map of quasi-coherent sheaves $\sO_X\to \sO_X[\beta^{-1}]$ induces an equivalence on connective covers. We claim that it suffices that the map $\sO_X\to \tau_{\ge 1}(\sO_X)^{\o -1}$, induced from $\beta$ by smashing with $\tau_{\ge 1}(\sO_X)^{\o -1}$, inducing an equivalence on connective covers. Indeed, if this holds, the same holds by induction for $\sO_X\to \tau_{\ge 0}(\sO_X)^{\o -n}$ for all $n\ge 0$, and since the $t$-structure on $\QCoh(X)$ on a geometric spectral stack $X$ is compatible with filtered colimits by \cite[Corollary 9.1.3.2]{SAG}, i.e.\ the functor $\tau_{\ge 0} : \QCoh(X)\to\tau_{\le 0}(\QCoh(X))$ commutes with them, we obtain the desired result by passing to $n\to\infty$ along $\varinjlim_n (\tau_{\ge 1}(\sO))^{\o -n}\simeq \sO_X[\beta^{-1}]$.

Thus we need to study the map $\sO_X\to \tau_{\ge 1}(\sO_X)^{\o -1}$. Since the connective cover is a right adjoint as such preserves limits, it follows that 
$$
\tau_{\ge 0}(\mathrm{fib}(\sO_X\to \tau_{\ge 1}(\sO_X)^{\o -1}))\to \tau_{\ge 0}(\sO_X)\to \tau_{\ge 0}(\tau_{\ge 1}(\sO_X)^{\o -1})
$$
is a fiber sequence in the prestable $\i$-category $\QCoh(X)^\mathrm{cn}$. It suffices to show that the left-most term is the zero object. Since that is (as are connective covers as well) a local question on $X$ for the fpqc topology, we may assume that $X$ is $n$-periodic. That means that $\tau_{\ge 1}(\sO_X)\simeq \Sigma^n(\mL)$ for some line bundle $\mL$ on $X$.  Using flatness of $\mL$, we obtain
\begin{eqnarray*}
\mathrm{fib}(\sO_X\to \tau_{\ge 1}(\sO_X)) &\simeq &\Sigma^{-1} (\mathrm{cofib}(\sO_X\to \tau_{\ge 1}(\sO_X)))\\
&\simeq &
\Sigma^{-1}(\mathrm{cofib}(\beta)\o_{\sO_X}\tau_{\ge 1}(\sO_X)^{\o -1})\\
&\simeq &\Sigma^{-n-1}(\pi_0(\sO_X)\o_{\sO_X}\mL^{\o -1})\\
&\simeq & \Sigma^{-n-1}(\pi_0(\mL)^{\o -1}),
\end{eqnarray*}
which is certainly concentrated in homotopical degrees $\le -(n+1) \le -1$, and so vanishes upon applying the connective cover functor $\tau_{\ge 0}$.
\end{proof}

\begin{corollary}\label{Cor heart per}
The map $\mX\to X$ induces an equivalence $\mX^\heart\simeq X^\heart$ upon underlying ordinary stacks.
\end{corollary}

\begin{remark}
It is not too difficult to show that a geometric non-connective spectral stack $\mX$ is of the form $\mX\simeq D_X(\beta)$ for an $n$-periodic spectral stack $X$ (which will then be its connective cover) if and only if there exists a line bundle $\mL$ on $\mX$ such that $\sO_\mX\simeq \Sigma^n(\mL)$. In that case, we have on the level of homotopy sheaves isomorphisms of graded $\sO_{X^\heart}$-algebras $\pi_*(\sO_X)\simeq \sO_{X^\heart}[\beta]$ and $\pi_*(\sO_\mX)\simeq \sO_{X^\heart}[\beta^{\pm 1}]$ for a generator $\beta$ in degree $n$.
\end{remark}

Tying together all the preceding discussion, we obtain the main result of this Section.

\begin{theorem}\label{defeorem}
Let $\mX$ be a non-connective geometric spectral stack, such that  $\tau_{\ge 0}(\mX)$ is a periodic spectral stack, and $\mX =\tau_{\ge 0}(\mX) - \mX^\heart  $. There exists a canonical diagram of non-connective spectral stacks
$$
\begin{tikzcd}
\mX\ar{r} \ar{d} & \tau_{\ge 0}(\mX)\ar{d} & \mX^\heart \ar{l}\ar{d}\\
 \Spec(S)\ar{r} & 
\textup{$\mathrm{CDiv}_\mathrm{eff}^\dagger$}
  & 
  D_\mathrm{univ}
  \ar{l},
\end{tikzcd}
$$
in which the top row consists of the usual connective cover and underlying ordinary stack maps, and
 both squares are pullback squares. 
\end{theorem}

In particular, the map $\mX\to\tau_{\ge 0}(\mX)$ may be identified with the `generic fiber', i.e.\ with $\beta$ invertible, while $\mX^\heart\to\tau_{\ge 0}(\mX)$ may be viewed as the `special fiber', i.e.\ with $\beta = 0$.
Since all the morphisms in sight are affine, this induces, together with Proposition \ref{QC is loc}, on the level of quasi-coherent sheaves, equivalences of $\i$-categories
\begin{equation}\label{sheaf level cartoon}
\QCoh(\mX)\simeq \QCoh(\tau_{\ge 0}(\mX))^{\mathrm{Loc}(\beta)},\qquad \QCoh(\mX^\heart)\simeq \Mod_{\sO_{\tau_{\ge 0}(\mX)}/\beta}(\QCoh(\tau_{\ge 0}(\mX))).
\end{equation}
That is (a quasi-coherent version of) the main $\tau$-deformation picture, as obtained by applying it to $\mX = \M$ (and passing to ind-coherent sheaves) in Section \ref{Section 1} above.

\begin{exun}
The affine non-connective spectral schemes, corresponding to the topological complex $K$-theory spectrum $\mathrm{KU}$ and its connective cover $\mathrm{ku}=\tau_{\ge 0}(\mathrm{KU})$, also fit into this paradigm. The underlying ordinary scheme map $\Spec(\mathbf Z)\to\Spec(\mathrm{ku})$ is an extended effective Cartier divisor, with the connected cover map $\Spec(\mathrm{KU})\to\mathrm{Spec}(\mathrm{ku})$ for its complement.
\end{exun}

\begin{remark}\label{Remark problems 2}
It may seem tempting to try to give a different   stack-level interpretation of the deformation picture, as
 manifests on the level of the $\i$-categories of quasi-coherent sheaves in \eqref{sheaf level cartoon} (or on the level of ind-coherent sheaves in Section \ref{section 6}), than the one provided by Theorem \ref{defeorem}. Such an interpretation might use the \textit{flat affine line} $\mathbf A^1_\flat = \Spec (S[t])$ and the multiplicative group $\G_m\subseteq \A^1_\flat$, corresponding to the polynomial $\E$-ring $S[t]:=S[\mathbf Z_{\ge 0}]$  (as opposed to the differentially smooth affine line $\A^1$ and the general linear group $\GL_1\subseteq \A^1$, based on the free $\E$-ring $S\{t\}\simeq S[\Omega^\i(S)]$ that featured in Lemma  \ref{nogolemma}). The main reason this may seem attractive, is the 
 standard equivalence of $\i$-categories with filtered spectra $\QCoh( \mathbf A_\flat^1/\mathbf G_m)\simeq \Sp^\mathrm{fil}$, proved in this context in \cite[Theorem 1.1]{Moulinos}. That might elucidate the central, albeit auxiliary, appearance of formal spectra throughout Section \ref{Section 1}.
 There are serious problems with this idea, however. In order to accommodate the example of  the spectral stacks $\tau_{\ge 0}(\M)$,  we would need to consider deformation diagrams over something akin to the quotient stack $\mathbf A_\flat^1[2]/\mathbf G_m$. The trouble arises in trying to make sense of
 the ``$2$-shifted flat affine line'' $\mathbf A_\flat^1[2]$. It should be something like $\Spec (S[\beta])$, where $S[\beta]$ is the graded ring spectrum with underlying graded spectrum $S[\beta]\simeq \bigoplus_{n \ge 0}\Sigma^{-2n}(S)$, as constructed in \cite[Section 3.4]{Lurie Rotation}. The issue is that $S[\beta]$ is only an $\mathbb E_2$-ring, instead of $\E$.  Due to the inherent difficulties of $\mathbb E_2$-spectral algebraic geometry via the `functor of points' approach\footnote{See however \cite{Francis} for a discussion of $\mathbb E_n$-spectral algebraic geometry, based on the \'etale topology and locally ringed $\i$-topoi.} , stemming from  pullbacks of affine $\mathbb E_2$-spectral schemes no longer being computed by relative smash products, we do not try pursuing this line of inquiry here. 
\end{remark}

\begin{remark}
Given the $\mathbb E_2$-difficulties we drew attention to in the previous Remark,
 one might wonder how come we did not encounter any of them so far. The reason is that the subtleties concerting the difference between the $\mathbb E_2$-spectral stacks (in some sense of what that should mean) $\mathbf A_\flat^1/\mathbf G_m$ and $\mathbf A_\flat^1[2]/\mathbf G_m$ may to an extent be circumvented on the level of the $\i$-categories of quasi-coherent (and consequently ind-coherent) sheaves. This is due to the $\i$-category of graded spectra $\mathrm{Sp}^\mathrm{gr}$ admitting an $\mathbb E_2$-monoidal self-equivalence $\fatslash : \Sp^\mathrm{gr}\simeq \Sp^\mathrm{gr}$ called \textit{shearing}, and given by
$
 X_*\mapsto \Sigma^{-2*}(X_*).
$
It exchanges the polynomial $\E$-ring $S[t]$, equipped with the usual grading, and the graded $\mathbb E_2$-ring $S[\beta]$. As such, shearing induces an equivalence of monoidal $\i$-categories
$$
\QCoh(\mathbf A_\flat^1/\mathbf G_m)\simeq \Sp^\mathrm{fil}\simeq \Mod_{S[t]}(\Sp^\mathrm{gr})\overset{\fatslash}{\simeq}\Mod_{S[\beta]}(\Sp^\mathrm{gr})\simeq \QCoh(\mathbf A_\flat^1[2]/\mathbf G_m),
$$
allowing us to pass between modules over $\QCoh(\mathbf A_\flat^1/\mathbf G_m)$ and $\QCoh(\mathbf A_\flat^1[2]/\mathbf G_m)$, by viewing both as filtered spectra. This is implicitly used in Section \ref{Section 1}, since the graded object $\Sigma^{2*}(\omega^{\o *})$, with its $S[t]$-module structure corresponding to the filtration, shears to the more traditional graded object $\omega^{\o *}$, encoding to the (strict) line bundle $\omega$ on the spectral stack $\tau_{\ge 0}(\M)$, equipped with a corresponding $S[\beta]$-module structure.
\end{remark}

\end{document}